\documentclass[11pt]{article}
\usepackage{amsmath,amssymb,amsbsy,amsfonts,latexsym,amsopn,amstext,cite,                                               amsxtra,euscript,amscd,bm}
\usepackage{latexsym}
\usepackage{amsmath}
\usepackage{amsfonts}
\usepackage{amssymb}
\usepackage{graphicx}
\usepackage{enumerate}
\usepackage[shortlabels]{enumitem}
\usepackage{bbm}
\usepackage{bm}
\usepackage{todonotes}
\usepackage[colorlinks,linkcolor=blue,anchorcolor=blue,citecolor=blue,backref=page]{hyperref}

\renewcommand*{\backref}[1]{}
\renewcommand*{\backrefalt}[4]{%
\ifcase #1 (Not cited.)%
\or        (p.\,#2)%
\else      (pp.\,#2)%
\fi}

\newenvironment{proof}{{\bf Proof:  }}{\hfill\rule{2mm}{2mm}}

\newtheorem{thm}{Theorem}[section]
\newtheorem{cor}[thm]{Corollary}
\newtheorem{defn}[thm]{Definition}

\newtheorem{lem}[thm]{Lemma}
\newtheorem{prop}[thm]{Proposition}
\newtheorem{rem}[thm]{Remark}
\newtheorem{conj}[thm]{Conjecture}
\newtheorem{que}[thm]{Question}
\newtheorem{claim}[thm]{Claim}
\newcommand{\ds} {\displaystyle}

\newcommand{\R}[0]{\mathbb{R}}
\newcommand{\T}[0]{\mathbb{T}}
\newcommand{\C}[0]{\mathbb{C}}
\newcommand{\N}[0]{\mathbb{N}}

\newcommand{\Z}[0]{\mathbb{Z}}

\newcommand{\1}[0]{\mathbbm{1}}
\newcommand{\K}[0]{\mathbb{K}}
\newcommand{\Kl}[0]{\mathcal{K}}
\newcommand{\B}[0]{\mathcal{B}}

\newcommand{\eE}[0]{\mathbb{E}}
\newcommand{\F}[0]{\mathcal{F}}

\newcommand{\Id}[0]{\textrm{Id}}
\newcommand{\M}{\mathcal{M}}
\newcommand{\A}{\mathcal{A}}
\newcommand{\pr}{\mathcal{P}}
\newcommand{\bmu}[0]{\bm \mu}

\newcommand{\bml}[0]{\bm \lambda}
\newcommand{\eps}[0]{\varepsilon}

\newcommand{\qid}[0]{\enspace\lower6pt\vbox{\hrule height2.6pt depth-
2pt\hbox{\vrule
width0.6pt\kern1pt\vbox{\kern1pt\phantom{j}\kern1pt}\kern1.5pt
\vrule width0.6pt}\hrule height2.6pt depth-2pt}}

\def\build#1_#2^#3{\mathrel{\mathop{\kern 0pt#1}\limits_{#2}^{#3}}}
\def\tend#1#2#3{\build\hbox to 12mm{\rightarrowfill}_{#1\rightarrow #2}^{ #3}}
\newcommand{\setdef}{\stackrel {\rm {def}}{=}}
\title{Sarnak's M\"{o}bius disjointness for dynamical systems with singular spectrum and dissection of M\"{o}bius flow}

\author{{\small {\MakeLowercase{el}} Houcein {\MakeLowercase{el}} Abdalaoui} \\
{\small Univ Rouen Normandie}\\ 
{\small  CNRS, Department of Mathematics, LMRS  UMR 6085 F-76000, Rouen, France} \\         
{\medskip} \\
{and} \\
{\medskip} \\
{\small Mahesh Nerurkar} \\
{\small Department of Mathematics} \\
{\small Rutgers University, Camden NJ 08102 USA}}

\date{}
    
\begin{document}
\maketitle
\renewcommand{\theequation}{\thesection.\arabic{equation}}
\numberwithin{equation}{section}

\medskip

\begin{abstract}
	We prove that Sarnak’s M\"{o}bius orthogonality conjecture holds for compact metric dynamical systems in which every invariant measure has singular spectrum. The proof relies on a computation by W. Veech concerning the correlation between the M\"{o}bius function and its square, together with the Rokhlin–Sinai machinery, which links positive entropy to the presence of a countable Lebesgue component in the spectrum. 
	
%We shall also present the proof of an unpublished theorem
%of Veech asserting that if the first projection is in the orthocomplement of the
%Pinsker sigma algebra of any quasi-generic measure in M\"{o}bius flow, then Sarnak M\"{o}bius
%disjointness conjecture holds. We notice that the converse holds.\\
We also present a proof of an unpublished theorem of Veech which asserts that if the first coordinate projection lies in the orthocomplement of the Pinsker $\sigma$-algebra of every quasi-generic measure for the M\"{o}bius flow, then Sarnak’s M\"{o}bius disjointness conjecture holds. Moreover, we observe that the converse implication is valid.

Among other consequences,
we obtain a simple proof of Matom{\"a}ki-Radziwi{\l}{\l}-Tao's theorem and Matom{\"a}ki-Radziwi{\l}{\l}'s
theorem  on the correlations of order two of the Liouville function without rate of convergence. We further present a M\"{o}bius-weighted ergodic theorem for weakly rigid operators due to Michael Lin.

We note that, aside from Davenport’s estimate, the only additional number-theoretic input in our argument is encoded in the Sarnak–Veech theorem. Moreover, we extend the Rokhlin–Sinai framework for  $\K$-systems so as to encompass the study of spectral properties of dynamical systems arising from number theory.

%Finally Appendix, authored by Michael Lin, is devoted to a M\"{o}bius-weighted ergodic theorem for weakly rigid operators
%where the main result in this paper plays a crucial role.

%Finally, is presented. In its proofs the main result in this paper plays a crucial role.

%abstract (150 word) We prove that Sarnak’s Möbius orthogonality conjecture holds for compact metric dynamical systems whose invariant measures have singular spectrum. The proof combines a computation of W. Veech on correlations between the Möbius function and its square with the Rokhlin–Sinai theory, which relates positive entropy to the presence of a countable Lebesgue component in the spectrum. We also establish an unpublished theorem of Veech showing that Sarnak’s Möbius disjointness conjecture is equivalent to the first coordinate projection lying in the orthocomplement of the Pinsker σ-algebra for all quasi-generic measures of the Möbius flow. As consequences, we obtain simplified proofs of the Matomäki–Radziwiłł–Tao theorem and of Matomäki–Radziwiłł’s result on second-order correlations of the Liouville function. Apart from Davenport’s estimate, the only number-theoretic input is the Sarnak–Veech theorem. We further extend the Rokhlin–Sinai framework for K-systems to dynamical systems arising in number theory and include a Möbius-weighted ergodic theorem for weakly rigid operators due to Michael Lin.

\vskip 1.0cm


\small

\end{abstract}

\pagenumbering{arabic}

%%%%%%%%%%%%%%%%%%

%if the sequence (a_n) satisfy for a large prime p and q the orthogonality of (a_{np}) and (a_{nq}) holds then the orthogonality holds between (a_n) and any multiplicative function say Möbius and Liouville.

%This is also the reason why it is not possible to obtain a  proof of the conjecture by only Daboussi-Katai-Bourgain-Sarnak-Ziegler criterion. Since, we need to verify that the conjecture holds for the constant function in any system and for the eventual  periodic sequence.

%Precisely , the PNT (versus Dirichelt PNT) is used to prove the conjecture holds for  dendrites which include interval maps and automatic sequence (result by Muller in Duke J.) and in the other case (irrational rotation, some class or rank one maps, some class of substitutions, horocycle flow, some class of rigid maps (Mariusz Radziwi{\l}{\l} and kanikowski), some distal system as group extension of rotation (Liu-Sarnak and other authors, some class of Kakutani sequence (Veech result), etc).

%Precisely , the PNT (versus Dirichelt PNT) is used to prove the conjecture holds for  dendrites which include interval maps and automatic sequence (result by Muller in Duke J.) and in the other case (irrational rotation, some class or rank one maps, some class of substitutions, horocycle flow, some class of rigid maps (Mariusz Radziwi{\l}{\l} and kanikowski), some distal system as group extension of rotation (Liu-Sarnak and other authors, some class of Kakutani sequence (Veech result), etc).

%%%%%%%%%%%%%%%%%

\section{Introduction}

The purpose of this article is to establish the validity of Sarnak's Möbius orthogonality conjecture for dynamical systems with singular spectra. 
More precisely, this means that Sarnak's conjecture holds for any compact metric dynamical system for which every invariant measure has a singular spectrum. 
We remark that in \cite{AD}, it was observed that the Chowla conjecture of order two implies that Möbius disjointness holds for any dynamical system with singular spectra. 
We also recall that Sarnak, in his seminal paper, states that the Chowla conjecture on the multiple correlations of the Möbius function 
implies Sarnak's Möbius disjointness conjecture. For more details about the connection between the Chowla conjecture and the Möbius orthogonality conjecture, we refer to \cite{AV}.

We now briefly describe our proof strategy. We first obtain a special case of the Chowla conjecture which gives a correlation between the Möbius function and its square (Theorem \ref{SNMV}). Next, using a computation of W. Veech, we establish a key proposition (Proposition \ref{VMV}). Another ingredient of the proof involves understanding Pinsker sigma-algebras for any potential spectral measure (which we call a potential spectral measure) for the 'Möbius flow' (i.e., for the shift dynamical system generated by the Möbius function). In this understanding, a 'generating partition' type result of Rokhlin and Sinai plays an important role. Our argument also uses the existence of 'the Chowla–Sarnak–Veech measure' (which we call a CSV measure) (Theorem \ref{SV}), and the observation of Veech that the first projection lies in the orthocomplement of the $L^2$-space of the Pinsker factor of the CSV measure. This, combined with the 'Rokhlin–Sinai machinery', allows us to describe the spectral type of the Möbius flow. Once this is done, the rest of the argument uses the notion of 'affinity between measures', also called the 'Hellinger method'. This method was developed by Below and Losert \cite{BeL} and may go back to Coquet–Mendès France–Kamae \cite{CMFK}.

The paper is organized as follows. In Section 2, we recall the basic ingredients of this method, along with results on the spectral measure associated with a sequence. We give the proof of our main theorem in Section 3.  In Section 4, we give some consequence. Finlay, in section 5, we present a proof of an unpublished theorem due to W. Veech (Theorem \ref{Veech}). In light of the proof of this theorem, we formulate our result in the language of spectral isomorphisms (Corollary \ref{spect-iso}).
Before formally beginning the proof, we summarize the known cases where this conjecture has been proved using various techniques. We remark that the ingredients used until now broadly fall into three groups, which we describe below.

\begin{enumerate}[label=(\arabic*)]
\item The first and foremost ingredient is the prime number theorem (PNT) and the Dirichlet PNT.

\item The result of Matom\"{a}ki-Radziwi{\l}{\l}-Tao on the validity of average Chowla of order
two \cite{MRT}. This was used to establish the conjecture for systems with discrete spectra.
It was used by el Abdalaoui-Lema\'{n}czyk-de-la-Rue  \cite{ALR}, Huang-Wang and Zhang
\cite{HWZ} , Huang-Wang-Ye  \cite{HZY}.

\item The Daboussi-Katai-Bourgain-Sarnak-Ziegler criterion, which says that if a sequence $(a_n)$
satisfy for a large prime $p$ and $q$, the orthogonality of $(a_{np})$ and $(a_{nq})$, then the
orthogonality holds between $(a_n)$ and any multiplicative function. In fact this criterion is also
based on PNT, that is, the PNT implies Daboussi-Katai-Bourgain-Sarnak-Ziegler criterion. We remark
that however the converse implication is false. For example, take the trivial sequence $a_n=1$.
This sequence does not satisfy the criterion. But the M\"{o}bius sequence is orthogonal to the
constant sequence $a_n=1$, $n \in \N$ is exactly the PNT. This criterion was used first by
Bourgain-Sarnak-Ziegler \cite{BSZ} and then by many other authors. Of course it is not possible
to prove the conjecture just using this criterion because one needs to verify the conjecture for
the constant function in any system and for the eventually periodic sequence, this is where
again PNT and Dirichlet PNT come into play.
\end{enumerate}

Our method does not use any of the above techniques. From number theory, we need Davenport's theorem, which ensures that the Möbius function is orthogonal to any rotation \cite{Da}, and the Sarnak–Veech theorem, which ensures the existence of the CSV measure. The rest relies on tools from positive entropy systems and spectral theory, by which we mean the Rokhlin–Sinai machinery and its spectral consequences combined with the Hellinger method. On the other hand, we shall strengthen, as corollaries to our result, some results proved in \cite{MR}, \cite{LGE}, \cite{FW} using the above techniques (see, e.g., Corollaries \ref{main-MRT} and \ref{FWLG}).
\medskip

\section{Definitions and tools}\label{Tools}

\medskip

 We start by recalling the definition of the M\"{o}bius function, which is intimately linked to the Liouville
 function, denoted by $\bml$. The latter function is defined to be $1$ if the number of prime factors
 (counting with multiplicity) is even, and $-1$ if not. The M\"{o}bius function is given by
 \begin{equation}\label{Mobius}
 	\bmu(n)= \begin{cases}
 		1 {\rm {~if~}} n=1; \\
 		\bml(n) {\rm {~if~}} n
 		{\rm {~is~square-free~}} ; \\
 		0 {\rm {~if~not.}}
 	\end{cases}
 \end{equation}
 We recall that $n$ is square-free if $n$ has no factor in the subset
 $\pr_2\setdef\big\{p^2/ p \in \pr\big\}$, where,
 as customary, $\pr$ denotes the set of prime numbers.
 
 \noindent A topological dynamical system is a pair $(X,T)$ where $X$ is a compact metric space and
 $T$ is a homeomorphism. The topological entropy $\textrm{h}_{\textrm{top}}(T)$ of $T$ is given by
 \[\textrm{h}_{\textrm{top}}(T) = \underset{\varepsilon\to 0}\lim \underset{n\to +\infty}\limsup \dfrac{1}{n}
 \textrm{log}~\textrm{sep}(n, T, \varepsilon),\]
 where, for $n$ an integer and $\varepsilon>0$, $\textrm{sep}(n, T, \varepsilon)$ is the maximal possible
 cardinality of an ($n, T, \varepsilon$)-separated set in $X$; this means that for every two distinct points
 of it, there exists $0\leq j<n$ with $d(T^j(x), T^j(y)) > \varepsilon$, where $T^{j}$ denotes the
 $j$-th iterate of $T$. For more details, we refer to \cite[Chap. 7]{W}.
 
 \noindent We are now able to state Sarnak's M\"{o}bius disjointness conjecture. It states that for
 any compact metric topological dynamical system $(X,T)$ with topological entropy zero, we have,
 for any $x \in X$ and any $f\in C(X)$,
 $$
 \frac{1}{N}\sum_{n=1}^{N}\bmu(n)f(T^nx) \tend{N}{+\infty}{}0\,.
 $$
 
 \noindent The Chowla conjecture on the correlation of the M\"{o}bius function states that, for any
 $r\geq 0$, $1\leq a_1<\dots<a_r$, $i_s\in \{1,2\}$ not all equal to $2$, we have
 \begin{equation}\label{cza}
 	\sum_{n\leq N}\bmu^{i_0}(n)\bmu^{i_1}(n+a_1)\cdot\ldots\cdot\bmu^{i_r}(n+a_r)={\rm o}(N).
 \end{equation}
 
 \noindent This conjecture implies a weaker conjecture stated by Chowla in \cite[Problem 56. pp.~95--96]{chowla}.
 
 \noindent The Chowla conjecture has a dynamical interpretation. Indeed, one can view $\bmu$ as a point
 in a compact space $X_3=\{0,\pm 1\}^{\mathbb{N}}$. As customary, we consider the shift map $S$ on
 it, and by taking the map $s~:~(x_n) \in X_3 \mapsto (x_n^2) \in X_2 \setdef \{0,1\}^{\mathbb{N}}$,
 it follows that the topological dynamical system $(X_2,S)$ is a factor of $(X_3,S)$. We denote by
 $\textrm{pr}_1$ the projection map onto the first coordinate, that is, $\textrm{pr}_1(x)=x_1$, where
 $x \in X_i,$ $i=2,3$.
 
 For a careful study of the Chowla conjecture, one also needs to introduce the notion of admissibility.
 The latter notion is due to Mirsky \cite{Ms} (see also \cite{Sarnak}).
 
 \smallskip
 
 \begin{defn}
 	(I) The subset $A \subset \N$ is admissible if the cardinality $t(p,A)$ of classes
 	modulo $p^2$ in $A$, given by
 	\[ t(p,A) \setdef \left|\bigl\{z\in\Z/p^2\Z: \exists n\in A, n=z\ [p^2]\bigr\}\right|, \]
 	satisfies
 	\begin{equation}
 		\label{eq:def-admissible}
 		\forall p \in \pr ,\ t(p,A)<p^2.
 	\end{equation}
 	In other words, for every prime $p$, the image of $A$ under reduction mod $p^2$ is a proper
 	subset of $\Z/p^2\Z$.
 	
 	\noindent (II) An infinite sequence $x=(x_n)_{n\in\N} \in X_3$ is said to be \emph{admissible}
 	if its \emph{support} $\{n\in\N : x_n \neq 0\}$ is admissible. In the same way, a finite block
 	$x_1\ldots x_N\in\{0,\pm 1\}^N$ is \emph{admissible} if $\{n\in\{1,\ldots,N\}: x_n \neq 0\}$
 	is admissible. In the same manner, we define the admissible sets in $X_2$.
 \end{defn}
 
 \smallskip
 
 \noindent For each $i = 2,3$, we denote by $X_{\A_i}$ the set of all admissible sequences in
 $X_i$. Since a set is admissible if and only if each of its finite subsets is admissible, and a
 translation of an admissible set is admissible, $X_{\A_i}$ is a closed and shift-invariant subset
 of $X_i$, \textit{i.e.} a subshift. We further have that $\bmu^2$ is an admissible sequence,
 and $X_{\A_3}=s^{-1}(X_{\A_2}),$ where $s$ is the square mapping. We denote by
 $\B(\A_i)$ the Borel $\sigma$-algebra generated by $\A_i$, $i = 2,3$.
 
 \smallskip
 
\subsection{\bf Some tools from spectral theory of dynamical systems }

Let $(X,\A,\mu,T)$ be a measurable dynamical system, that is, $(X,\A,\mu)$ is a
probability space and  $T$ is an invertible, bi-measurable map which preserves 
 $\mu$,
i.e. $\mu(T^{-1}(A))=\mu(A)$, for every $A \in \A$. The dynamical system
is ergodic if every $T$-invariant set is trivial:
$\mu(T^{-1}(A) \triangle A)=0 \Longrightarrow \mu(A) \in \{0,1\}$. Transformation
$T$ induces an operator $U_T$ in $L^p(X)$ via $f \mapsto U_{T} (f)= f \circ T$
called Koopman operator. For  $p=2$ this operator is unitary and its spectral
resolution induces  a spectral decomposition of $L^{2} (X)$ \cite{P}:
\[
L^2(X)=\bigoplus_{i=0}^{+\infty}C(f_i) {\rm {~~and~~}}
\sigma_{f_1}\gg \sigma_{f_2}\gg\cdots
\]
where
\begin{itemize}
	\item $\{f_{i} \}_{i=1}^{+\infty}$ is a family of functions in $L^{2} (X)$;
	\item $C(f)\setdef \overline{\rm {span}}\{U_T^n(f): n \in \Z\}$ is the cyclic
	space generated by $f \in L^2(X)$;
	\item  $\sigma_f$ is the {\em spectral measure} on the circle generated by $f$
	via the Bochner-Herglotz relation
	\begin{equation}\label{fspmeasure}
	\widehat{\sigma_f}(n)=<U_T^nf,f>=\int_X f \circ T^n(x) \overline{f}(x)d\mu (x);
	\end{equation}
	\item for any two measures on the circle $\alpha$ and $\beta$,
	$\alpha \gg \beta$ means $\beta $ is absolutely continuous  with respect to
	$\alpha$: for any Borel set $A$, $\alpha(A)=0 \Longrightarrow \beta(A)=0$.
	The two measures $\alpha$ and  $\beta$ are equivalent if  and only if
	$\alpha\gg\beta$ and $\beta\gg\alpha$. We will denote measure equivalence by
	$\alpha \sim \beta$.
\end{itemize}
The spectral theorem ensures this spectral decomposition is unique up to
isomorphisms.
The {\em maximal spectral type} of $T$ is the equivalence class of the Borel
measure $\sigma_{f_1}$. The multiplicity function
$\M_{T} : \T \longrightarrow \{1,2,\cdots,\} \cup \{+\infty\}$ is defined
$\sigma_{f_1}$ a.e. and
\[
\M_T(z)=\ds \sum_{j=1}^{+\infty}\1_{Y_j}(z), \quad {\rm  where}, \ Y_1=\T \ {\rm and}\
Y_j={\rm{~supp~}}\frac{d\sigma_{f_j}}{d\sigma_{f_1}} \quad \forall j \geq 2.
\]
An integer $n \in \{1,2,\cdots,\} \cup \{+\infty\}$ is called an essential value of
$M_T$ if $\sigma_{f_1}\{ z \in \T : M_T(z)=n\}>0$. The multiplicity
is uniform or homogeneous if there is only one essential value of $M_T$.
The essential supremum of $M_T$ is called the maximal  spectral multiplicity of $T$.
The map $T$
\begin{itemize}
	\item has simple spectrum if $L^2(X)$ is a single cyclic space;
	\item has discrete spectrum if $L^2(X)$ has an
	orthonormal basis consisting of eigenfunctions of $U_T$,
	(in this case $\sigma_{f_1}$ is a discrete measure);
	\item has Lebesgue spectrum if $\sigma_{f_1}$ is equivalent 
	to the Lebesgue measure. It has absolutely continuous (or singular)
	spectrum if $\sigma_{f_1}$ is absolutely continuous (or singular)
	with respect to the Lebesgue measure.
\end{itemize}

\smallskip

\begin{defn} \label{defsingular} The {\em reduced spectral type} of the dynamical system is its
spectral type on the $L_0^2 (X)$ - the space of square integrable functions with zero mean.
Two dynamical systems are called {\em {spectrally disjoint}} if their reduced spectral types are
mutually singular.
\end{defn}

\smallskip

\noindent{}As customary, for any sub-$\sigma$-algebra $\mathcal{C}$ of $\A$, we denote by $L^2 (X, \mathcal{C}, \mu)^{\perp}$ the orthocomplement to the subspace $L^2 (X, \mathcal{C}, \mu)$ in $L^2(X,\A,\mu)$.

%\begin{lem}[Rokhlin \cite{Rh},\cite{P}]\label{RH} Let $(X,\A,\mu,T)$ be a measure theoretic dynamical system. If $\mathcal{C}$ is a proper non-atomic sub-$\sigma$-algebra of $\A$,  then  $L^2 (X, \mathcal{C}, \mu)^{\perp}$ is infinite-dimensional.
%\end{lem}

%\noindent As a consequence, of Lemma \ref{RH} and the proof of a theorem due to Rokhlin and Sinai, (see Lemma (14.1-14.3) in \cite{Rh}), we have the following proposition.

\smallskip

%\begin{prop}[Rokhlin-Sinai \cite {Rh}]\label{P}Let $(X,\A,\mu,T)$ be a measure theoretic dynamical
%system on a Lebesgu space and $\Pi(T)$ be its Pinsker $sigma$-algebra. Suppose $h_\mu (T)u)$
%is positive. Then

%\noindent (i) $L^2 (X,\Pi (T),\mu )^\perp$-the orthocomplement to the subspace
%$L^2(X,\Pi (T),\mu)$ in $L^2(X,\A,\mu)$ is infinite dimensional and

%\noindent (ii) $U_T$ has countable Lebesgue spectrum on $L^2 (X,\Pi (T),\mu )^\perp$.
%\end{prop}

%\smallskip

\noindent  We appeal to the following fundamental theorem of Rokhlin–Sinai, which is central to their framework. %and to the proof of Proposition \cite{P}. 
It also provides the cornerstone of the proof of our main result.

\begin{thm}[Rokhlin-Sinai \cite {Rh}]\footnote{See also \cite[Theorem 12, p.69]{P}, \cite[Theorem 18.9, p.323]{G}.} \label{Q}Let $(X,\A,\mu,T)$ be a measure theoretic dynamical
system on a Lebesgue space and $\Pi(T)$ be its Pinsker $\sigma$-algebra. Then, there exists a sub-$\sigma$-algebra $\F \subset \A$ such that $T^{-1}\F \subset \F$, $\bigvee_{n=0}^{+\infty}T^n \F=\F_{+\infty}=\A$ (i.e., the $\sigma$-algebra generated by $\{T^nF\ |\ F\in \F\,,\ n\in \{0\}\cup \N\}$ is $\A$), and $\bigcap_{n=0}^{+\infty}T^{-n}\F=\F_{-\infty}=\Pi(T)$.
\end{thm}
\noindent{} As is customary, if the sequence of $\sigma$-algebras $(\mathcal{F}_n)$ is increasing, we write  
\[
\mathcal{F}_n \nearrow \mathcal{F}_{+\infty},
\qquad \text{where } \mathcal{F}_{+\infty} := \sigma\!\left(\bigcup_{n} \mathcal{F}_n\right)=\bigvee_{n=0}^{+\infty}\F_n.
\]
Similarly, if the sequence of $\sigma$-algebras $(\mathcal{F}_n)$ is decreasing, we write  
\[
\mathcal{F}_n \searrow \mathcal{F}_{-\infty},
\qquad \text{where } \mathcal{F}_{-\infty} := \bigcap_{n} \mathcal{F}_n.
\]

Let us emphasize that if the measure-theoretic entropy of the dynamical system  $(X,\A,\mu,T)$ is positive (i.e. $h_\mu(T)>0$),  then the sub-\(\sigma\)-algebra  $\F \subset \A$  is non-trivial. We recall that  $h_\mu(T)$ is given by 

$$h_\mu(T)=\sup_{\mathcal{P}, H(P)<+\infty}h_\mu(T,\mathcal{P}),$$
where 
$$H(P)=-\sum_{A \in \mathcal{P}}\mu(A)\log(\mu(A)),$$
and 
$$h_\mu(T,\mathcal{P})=\lim \frac{1}{n}H_\mu(\bigvee_{i=0}^{n-1}T^{-i}\mathcal{P})=
\inf_{n \geq 1}H_\mu(\bigvee_{i=0}^{n-1}T^{-i}\mathcal{P}).$$
For each $n \in \Z$, we denote by \(H_n\) the Hilbert space 
\(L_0^2(X, T^n \mathcal{F}, \mu)\), and we define
\[
H_n \ominus H_{n-1}
= \left\{ f \in H_n \ :\ \forall g \in H_{n-1},\ \langle f, g \rangle = 0 \right\}.
\]
In this setting, we have the following theorem.

\smallskip

\begin{thm}[Rokhlin-Sinai \cite{CFS}]\label{CFS-F} If the measure-preserving transformation $T$ possesses a sub-\(\sigma\)-algebra  $\F \subset \A$ such that $\F \subsetneq T \F $, then the unitary operator $U_T$ has countable Lebesgue spectrum on $ \ds \bigoplus_{n \in \Z}H_n$.
\end{thm}
For the proof of Theorem \ref{CFS-F}, we refer to \cite[Chap. 13, pp. 338-339]{CFS}. As a consequence, of Theorem \ref{Q} and Theorem \ref{CFS-F}, we have the following Proposition due to Rokhlin and Sinai, (see Lemma (14.1-14.3) in \cite{Rh}).

\smallskip

\begin{prop}[Rokhlin-Sinai \cite {Rh}]\label{P}Let $(X,\A,\mu,T)$ be a measure theoretic dynamical
system on a Lebesgue space and $\Pi(T)$ be its Pinsker $\sigma$-algebra. Suppose $h_\mu (T)$
is positive. Then

\noindent (i) $L^2 (X,\Pi (T),\mu )^\perp$-the orthocomplement to the subspace
$L^2(X,\Pi (T),\mu)$ in $L^2(X,\A,\mu)$ is infinite dimensional and

\noindent (ii) $U_T$ has countable Lebesgue spectrum on $L^2 (X,\Pi (T),\mu )^\perp$.
\end{prop}

\medskip

\subsection{\bf Spectral measure of a sequence}\label{sequence-spectral-measure}

\medskip

The notion  of  spectral measure for sequences was introduced by Wiener in 1933, (see \cite{Wiener}).
Therein, he defines a space $\mathcal{W}$ of complex bounded sequences $a =(a_{n})_{n \in \N}$ such that
\begin{equation}\label{Sspace}
\lim_{N \longrightarrow +\infty}
\frac{1}{N}\sum_{n=0}^{N-1}a_{n+k}\overline{a}_{n}=F(k)
\end{equation}
exists for each integer $k \in \N$. The sequence $F(k)$ can be extended to negative
integers by setting
\[
F(-k)=\overline{F (k)}.
\]
It can be further seen that $F$ is positive definite on $\Z$ and therefore by the
Herglotz-Bochner theorem, there exists a unique positive finite measure $\sigma_a$ on
the circle $\T$ such that the Fourier coefficients of $\sigma_a$ are given by the sequence $F$, that is,
\[
\widehat{\sigma_a}(k)\stackrel{\rm{def}}{=}
\int_{\T} e^{-ikx} d\sigma_{a }(x) = F(k).
\]
The measure $\sigma_a $ is called the {\em spectral measure of the sequence $a$}.\\

\smallskip

\begin{rem}\label{IG}\
\begin{enumerate}[(I)]
\item We mention that according to Chowla conjecture, the spectral measure of the M\"{o}bius function
is the Lebesgue measure.

\item Given a topological dynamical system $(X,T)$, $f\in C(X)$ and $x\in X$, there is a natural connection
between the spectral measure of the dynamical sequence $n\mapsto f(T^{n}x)$ and the spectral type of
the dynamical system. Indeed, for a uniquely ergodic system $(X,T,\mu)$, this sequence belongs to the
Wiener space $\mathcal{W}$ and its spectral measure is exactly the spectral measure of the function $f$
\[
\widehat{\sigma}_f(k) = <U^{k}_T(f),f>=\int f \circ T^{k}(x). \overline{f(x)} d\mu(x)\,.
\]
In the general situation, the following notion of a `(set of) potential spectral measures' of $x$ captures
the appropriate concept.
\end{enumerate}
\end{rem}

\smallskip

\begin{defn} (1) Given a compact metric dynamical system $(X,T)$ and $x\in X$, let $\mathcal{I}_T(x)$
be the weak-star closure of the set $\big\{\frac{1}{N}\sum_{n=1}^{N}\delta_{T^nx}\ |\  N\in \N \big\}$).
Note that any $\nu_x \in  \mathcal{I}_T (x)$ is $T$-invariant. Each such $\nu_x$ will be called a
potential measure determined by $x$.

\noindent (2) For each such potential measure $\nu_x$, there is a subsequence $(N_{\ell})$ such that
for any $k\in \Z$ and for any $f\in C(X)$ we have,
\begin{equation}
\lim \frac{1}{N_{\ell}}\sum_{n=1}^{N_{\ell}}\delta_{T^nx}(f \circ T^k.\overline{f} ) =
\lim \frac{1}{N_{\ell}}\sum_{n=1}^{N_{\ell}} f(T^{n+k}x) \overline{f}(x) = \int f \circ T^k(y)
\overline{f}(y)d\nu_x =\widehat{\sigma_{f,\nu_x}}(k)\, ,
\end{equation}
where $\sigma_{f,\nu_x}$ is the spectral measure associated to the vector $f$ in the Hilbert space
$L^2(X,\nu_x)$. Such a spectral measure $\sigma_{f,\nu_x}$ will be called a potential spectral
measure determined by $x$, corresponding to the function $f$. 
\end{defn}

\medskip

\subsection{\bf Correlations and Affinity}

The notion of `affinity' was introduced and studied in a series of papers by Matusita
\cite{Matusita1},\cite{Matusita2},\cite{Matusita3} and  it is also called Bhattacharyya coefficient
\cite{Bhatta}. It has been widely used in statistics literature to quantify the similarity between two
probability distributions. The affinity  between two finite measures is defined by the integral of the
corresponding geometric mean. Let $\mathcal{M}_1(\T)$ be a set of probability measures on
the circle $\T$ and $\eta,\nu \in \M_1(\T)$ two measures in this space. There exists a probability
measure $\lambda$ such that $\eta$ and $\nu$ are absolutely continuous with respect to $\lambda$,
(take for example $\lambda=\frac{\eta+\nu}2$). Then the {\em affinity} between $\eta$ and
$\nu$ is defined by
\begin{equation}\label{affinity}
G(\eta,\nu)=\ds \int \sqrt{\frac{d\eta}{d\lambda}. \frac{d\nu}{d\lambda}} d\lambda.
\end{equation}
This definition does not depend on $\lambda$. Affinity is related to the Hellinger distance which is 
defined as
\[
H(\eta,\nu)=\sqrt{2(1-G(\eta,\nu))}.
\]
As a consequence of Cauchy-Schwarz inequality we have the following,
\[
0 \leq G(\eta,\nu) \leq  1.
\]
\begin{rem}\
\begin{enumerate}[(I)]
\item The definition of affinity can be extended to any pair of positive finite measures by normalizing them.

\item It is an easy exercise to see that $G(\eta,\nu)=0$ if and only if $\eta$ and $\nu$
are mutually singular (denoted by $\eta \bot \nu$):
this means that there exists a pair of disjoint Borel sets $A$ and $B$ such that $\eta$ is concentrated on $A$ and
$\nu$ is concentrated on $B$, ( a measure $\rho$ is concentrated on a Borel set $E$ if
$\rho(F)=0$ if and only if $ F \cap E= \emptyset$). Similarly, $G(\eta,\nu)=1$ holds if
and only if $\eta$ and $\nu$ are equivalent. %$\eta\ll \nu$ and $\nu\ll \mu$.
Affinity can be used to compare sequences of measures via the
following theorem. 
\end{enumerate}
\end{rem}

\begin{thm}[Coquet-Kamae-Mand\`es-France  \cite{CMFK}]\label{coquet-france}
	Let $(P_n)$ and $(Q_n)$ be two sequences of probability measures
	on the circle, weakly converging to the probability measures $P$ and $Q$
	respectively. Then
	\begin{equation}\label{limsup}
	\limsup_{n \longrightarrow +\infty } G(P_n,Q_n) \leq G(P,Q).
	\end{equation}
\end{thm}

\begin{rem}
As in the case of the affinity, this result can be generalized
to any sequence of positive non-trivial finite measures $P_{n}$, $Q_{n}$ on a compact space converging weakly to two positive non trivial finite measures $P$ and $Q$.
\end{rem}

\noindent We aim to use the notion of affinity, together with Theorem~\ref{coquet-france}, 
to estimate the orthogonality properties of pairs of sequences in the Wiener space 
$\mathcal{W}$ (as defined in Subsection~\ref{sequence-spectral-measure}). 
To this end, we replace the sequence of Fourier coefficients 
\[
\frac{1}{n}\sum_{j=0}^{n-1} g_{j+k}\,\overline{g_{j}}
\]
by a sequence of finite positive measures on the unit circle. 

\noindent For any $g \in \mathcal{W}$, we define the sequence of measures
\begin{equation}\label{rhodef}
d\sigma_{g,n}(x) = \rho_{g,n}(x)\,\frac{dx}{2\pi}, 
\qquad \text{where} \quad
\rho_{g,n}(x) = \left|\frac{1}{\sqrt{n}}\sum_{j=0}^{n-1} g_j e^{ijx}\right|^2.
\end{equation}

With this definition $\sigma_{g,n}$ defines a finite positive measure on the
circle. Moreover  we have the relation
\[
\frac{1}{n}\sum_{j=0}^{n-1} g_{j+k}\overline{g_{j}}=
\int_{0}^{2\pi } e^{-ikx} d\sigma_{g,n} (x) \ + \ \Delta_{n,k}  =
\widehat{\sigma }_{g,n} (k)+ \Delta_{n,k},
\]
where
\[
\left| \Delta_{n,k}  \right|=
\left| \frac{1}{n}\sum_{j=n-k}^{n-1} g_{j+k}\overline{g_{j}}  \right| \leq \
\frac{k}{n} \sup_j |g_{j}|^{2} \tend{n}{+\infty}{} 0.
\]
Taking the limit we have
\[
\widehat{\sigma }_{g} (k) = \lim_{n\to\infty}
\frac{1}{n}\sum_{j=0}^{n-1} g_{j+k}\overline{g_{j}}= \lim_{n\to\infty}
\widehat{\sigma }_{g,n} (k),
\]
so the sequence of measures $(\sigma_{g,n})_{n \in \N}$ converges weakly to $\sigma_{g}$.\\

\noindent It may be possible that the bounded sequence $(g_{n})_{n \in \N}$ does not belong to the
Wiener space $\mathcal{W}$ (see eq. \eqref{Sspace}) but
we can always extract a subsequence $(n_r)$ in \eqref{Sspace} such that
\[
\lim_{r \rightarrow \infty}
\frac{1}{n_r}\sum_{j=0}^{n_r-1}g_{j+k} \overline{g_{j}}
\]
exists for each $k\in \N$.
In fact, consider the sequence of finite positive measures $(\sigma_{g,n})_{n \in \N}$ on the circle defined in \eqref{rhodef}. These measures are all finite and $\sigma_{g,n} (\T)\leq  \|g\|_\infty^{2}
= \sup_j |g_{j}|^{2}$, for all $n$. Therefore they all belong the ball $B(0,\|g\|_\infty^{2})$
centered at $0$ with radius $\|g\|_\infty^{2}$ in the  set of measures on the circle. This subset
is compact so there exists a subsequence $(n_r)$ such that the sequence of probability measures
$(\sigma_{g ,n_r})_{r \in \N}$ converges weakly to some probability measure $\sigma_{g,(n_r)}$.
The measure $\sigma_{g ,(n_r)}$ is called the spectral measure of the sequence $g$ along the
subsequence $(n_r)$.\\

\noindent We will also need the following result whose proof follows from Theorem \ref{coquet-france}.
\begin{cor}\cite{BeL}\label{BL} Let $g = (g_{n})_{n\in \N},
h= (h_{n})_{n\in \N} \in \mathcal{W}$ be two non trivial sequences i.e.
$\widehat{\sigma}_{g} (0)>0$ and $\widehat{\sigma}_{h} (0)>0$. Then
\begin{equation}\label{BL-log}
\limsup_{n\to \infty} \left | \frac{1}{n}  \sum_{j=1}^{n}g_{j}\overline{h}_{j}\right|
\leq \sup \big\{G(\sigma_{g,(n_r)} ,\sigma_{h,(m_r)})\big\}.
\end{equation}
where the supremum on the right-hand side is taken over all subsequence $(n_r)$, $(m_r)$ for
which the spectral measures exists.	
\end{cor}

\smallskip

\begin{rem}We warn the reader that the inequality \eqref{BL-log} is, in general, strict. 
Indeed, as observed by Below and Losert~\cite{BeL}, the spectral measure of the sequence 
$g(n) = (-1)^{\lfloor \log(n) \rfloor}$ is the Dirac measure at $1$, yet
\begin{align}\label{Am}
\lim_{n \to \infty} \left| \frac{1}{n} \sum_{j=1}^{n} g_j \overline{z}^{\,j} \right| = 0,
\end{align}
for every $z \in \mathbb{T}$.

We also mention that there exists a class of counterexamples arising in number theory, 
such as the sequences
\[
g(n) = n^{i \log n}, \qquad g(n) = n^{it}, \quad t \neq 0, \text{ etc.}
\]
To see this, note that for any $k \in \mathbb{Z}$,
\[
\frac{1}{n} \sum_{j=1}^{n} g_{j+k} \overline{g_j} 
= \frac{1}{n} \sum_{j=1}^{n} 
   e^{\,i\big( (\log(j+k))^2 - (\log j)^2 \big)}.
\]
Moreover,
\[
(\log(j+k))^2 - (\log j)^2
= \log\!\Big(1 + \frac{k}{j}\Big)\, \log j \,
   \Big(1 + \log\!\Big(1 + \frac{k}{j}\Big)\Big),
\]
and by Taylor’s theorem,
\[
\log\!\Big(1 + \frac{k}{j}\Big)\, \log j \,
   \Big(1 + \log\!\Big(1 + \frac{k}{j}\Big)\Big)
   \sim \frac{k}{j}\,\log j\,\Big(1 + \frac{k}{j}\Big).
\]
Therefore,
\[
\lim_{n \to \infty} 
   \frac{1}{n} \sum_{j=1}^{n} g_{j+k} \overline{g_j}
   = \lim_{n \to \infty} 
   \frac{1}{n} \sum_{j=1}^{n} 
   e^{\,i\big((\log(j+k))^2 - (\log j)^2\big)}
   = 1.
\]

By applying van der Corput’s inequality~\cite[Theorem~2.7, p.~17]{KN}, 
one can verify that \eqref{Am} still holds.  

Finally, we note that this class of sequences forms a subclass of a broader family 
containing the bounded Besicovitch sequences. This larger class was introduced 
by Bertrandias and is known as the class of $\mathcal{M}^1$-almost periodic sequences 
(see~\cite[pp.~69–73]{Bert}). We also point out that the dynamical behavior of such 
sequences is quite different from that of the M\"{o}bius function.

\end{rem}

\medskip

\section{The main result}In this section, we begin by stating our main result (Theorem~\ref{main}) and by developing the necessary tools, namely Theorem~\ref{SNMV} and Proposition~\ref{VMV}.

\smallskip

\begin{thm}\label{main}Let $(X,T)$ be a dynamical system such that for any invariant measure
the spectrum is singular. Then, the M\"{o}bius disjointness holds.
\end{thm}

\smallskip

\noindent {\bf STEP I: A consequence of Rauzy's result} : First we remark that in what follows,
only the forward orbit of points in the shift spaces are involved. So, how one extends
a sequence to a bisequence is not very crucial as long as it is done in a certain canonical way, 
(e.g. $x_{-n} = x_n$). With this observation in place, we proceed to prove the following special case of the strong Sarnak–Möbius conjecture~\cite{Sarnak}.

\smallskip

\begin{thm}\label{SNMV} For any continuous function $f \in C(X_{\mathcal{A}_2})$, for any
$\theta \in \R$, we have
$$
\frac{1}{N}\sum_{n=1}^{N}\bmu(n)f(\bmu^2(n)) e^{in\theta} \tend{N}{+\infty}{}0\,.
$$
\end{thm}

\smallskip

The following corollary follows immediately.

\begin{cor} Let $k \geq 2$ and $a_1, \cdots,a_k$ be distinct non-negative integers and
$\theta \in \R$. Then
$$
\lim\limits_{N\to \infty}\frac {1}{N}\sum_{n \leq x}\bmu(n)\bmu^2(n+a_1)\cdots\bmu^2(n+a_k)e^{in\theta} =0\,.
$$
\end{cor}
\begin{proof} This follows by taking $f(x)=\textrm{pr}_{a_1}(x) \cdots \textrm{pr}_{a_k}(x)$,
for $x \in X_{\A_2}$.
\end{proof}

\smallskip

\begin{rem} This corollary can be viewed as a `non-quantitative version' of a recent result
due to R. Murty \& A. Vatwani \cite{MV}, (which is a generalization of Davenport's estimate \cite{Da}).
We emphasize that the above result relies only on ergodic and dynamical arguments.
\end{rem}
\smallskip

Now we show that Step I, (i.e. the proof of Theorem \ref{SNMV}) is a consequence of a
result of Rauzy, which asserts that the sequence $n\mapsto \bmu(n)^2$ is Besicovitch almost periodic.
We now recall the notion of Besicovich almost periodicity and give arguments that slightly
generalizes Rauzy's result.

\smallskip

\begin{defn}\ \\(1) A map $\phi~:~\N \longrightarrow \C$ is  Besicovitch almost periodic
if given $\epsilon>0$, there exists a trigonometric polynomial $P \equiv P_\epsilon$ given
by $P(n)=\sum_{k=1}^{M}c_ke^{i \alpha_k n}$, $n \in \N$, where $\alpha_k \in \R$,
for $1 \leq k \leq M$, such that
$$
\big\|\phi-P\big\|_{B_1} \setdef \limsup\limits_{N\to \infty}\frac {1}{N}\sum_{t=1}^{N}|\phi(t)-P(t)|<\epsilon\,.
$$

\noindent (2) Let $(X,T)$ be a compact metric topological dynamical system. A point $x_0\in X$
is a  Besicovitch almost periodic point if the sequence $n\mapsto f(T^n(x_0))$ is a  Besicovitch
almost periodic for each $f\in C(X)$.
\end{defn}

\begin{rem}\label{Theta}Let  $\psi~:~\N \longrightarrow \C$ be a  Besicovitch almost periodic
function. Then for any $\theta \in \R$ the map $\psi_\theta$ is also a Besicovitch almost periodic
function, where $\phi_\theta(n)=\psi(n)e^{i n \theta}$.

To see this, given $\epsilon>0$, let $P(n)=\sum_{k=1}^{M}c_ke^{i \alpha_k n}$ be a
trigonometric polynomial such that \linebreak $\big\|\psi-P\big\|_{B_1}<\epsilon$. Let
$Q(n)=P(n)e^{in\theta}=\sum_{k=1}^{M}c_ke^{i (\alpha_k+\theta) n}$. Note that $Q$ is also
a trigonometric polynomial and 
\begin{align*}
    \big\|\psi_\theta-Q\big\|_{B_1} &= \limsup\limits_{N\to \infty}\frac {1}{N}\sum_{t=1}^{N}|\psi_\theta(t)-Q(t)|\\
    &= \limsup\limits_{N\to \infty}\frac {1}{N}\sum_{t=1}^{N}\big|\psi(t)e^{i t \theta}-P(t)e^{i t \theta}\big| \\
    & = \big\|\psi-P\big\|_{B_1}<\epsilon.
\end{align*}
\end{rem}

\smallskip

The following result provides a sufficient condition for a point in a dynamical system to be a Besicovitch point (see~\cite{AM} for a proof).

\begin{lem} Let $(X,T,\mu)$ be a compact metric dynamical ergodic system with discrete spectrum.
Let $x_0\in X$ be a $\mu$-generic point. Then $x_0$ is a  Besicovitch point. 
\end{lem}

\smallskip

Now we recall the results of Mirsky \cite{Ms} and Sarnak-Cellarosi-Sinai (a proof appears in \cite{CS}.).

\smallskip

\begin{prop} (1) [Mirsky's theorem] The point $\bmu^2$ is a generic  point for the `square-free' dynamical system 
$(X_{\A_2},S,\nu_M)$, where $\nu_M$ is the Mirsky measure.

\noindent (2) [Sarnak-Cellarosi-Sinai's theorem] The dynamical system $(X_{\A_2},S,\nu_M)$ has discrete spectrum.
\end{prop}

\smallskip

\noindent{} As a consequence of the preceding proposition and lemma, we obtain a generalization of an observation due to Rauzy~\cite[p.~99]{R}, who showed that the sequence $n \mapsto \mu^2(n)$ is Besicovitch almost periodic. This result was later extended to a class of multiplicative functions by Daboussi~\cite{Dabo}.

\smallskip

\begin{prop} Consider  the dynamical system $(X_{{\mathcal{A}}_2},S,\nu_M)$. Then,
for any continuous map \linebreak $f : X_{\mathcal{A}_2} \to \C$, the map
$n \mapsto f(S^n{\bmu^2})$ is Besicovitch almost periodic.
\end{prop}

\smallskip

We now note that M\"{o}bius disjointness holds for Bohr almost periodic sequences, as a consequence of Davenport’s estimate \cite{Da}. An approximation argument (as in Remark~\ref{Theta}) extends this property to Besicovitch almost periodic functions. Hence, we obtain the following result.
\smallskip

\begin{lem}\label{BDav} Let $\{b_n\}$ be a Besicovitch almost periodic sequence. Then 
$$
\frac{1}{N}\sum_{k=1}^{N}\bmu(k)b_k \tend{N}{+\infty}{}0\,.
$$
\end{lem}

\smallskip

\noindent Now, the proof of Theorem \ref{SNMV} follows by putting all of these observations together.

\medskip

\noindent {\bf Proof of Theorem \ref{SNMV}}

\begin{proof} Let $f\in C(X_{\A_2})$. Then the map $n \mapsto f(S^n(\bmu^2))$ is Besicovitch
almost periodic and so is the map $n \mapsto f(S^n(\bmu^2))e^{in\theta}$. Thus, by Lemma
\ref{BDav},
$$
\frac{1}{N}\sum_{k=1}^{N}\bmu(k)f(S^k(\bmu^2))e^{in\theta}\tend{N}{+\infty}{}0\,.
$$
\end{proof}

\smallskip

\noindent {\bf STEP II: A consequence of a computation due to W. Veech} :

\smallskip

\begin{rem}$~$	
First we need to recall a theorem, (which we shall refer to as the Sarnak-Veech theorem)
and then we present a  crucial computation due to W. Veech \cite{VeechNotes2}. This
Sarnak-Veech's theorem was announced in \cite{Sarnak} and Veech gave an unpublished
proof in \cite{VeechNotes2}, (see \cite{AV} where Veech's proof is presented).  Sarnak-Veech's
theorem states that there exist a unique ergodic admissible measure $\eta_{M}$ on
$X_{\mathcal{A}_3}$ such that the factor $(X_{\A_2},S, \nu_{M})$ is the `Pinsker factor'
of the dynamical system $(X_{\mathcal{A}_3}, S,\eta_{M})$. This means that the Pinsker
sigma algebra $\Pi(S)$ of this dynamical system is given by $\Pi(S)=s^{-1}(\mathcal{B}(\A_2))$. 

In the following, we recall the definition of an admissible probability measure.
\end{rem}

\smallskip

\begin{defn}A probability measure $\eta_m$ on $X_{\A_3}$ is admissible if 
\begin{enumerate}[(i)]
\item $S\eta_m=\eta_m$, that is, $\eta_m(S^{-1}A)=\eta_m(A)$, for each Borel set $A \subset X_{\A_3}$.
\item $s(\eta_m)=\nu_M$, and
$$
\int_{X_{\A_3}}\prod_{ a\in A} \textrm{pr}_a(x) \prod_{b \in B}\textrm{pr}_b^2(x) d\eta_M(x)=0\,,
$$
for any $A \neq \emptyset.$ and $B$ finite sets of $\N$.
\end{enumerate}
\end{defn}	

\smallskip

Then, the Sarnak-Veech's theorem states the following.

\smallskip 

\begin{thm}[Sarnak-Veech's theorem on M\"{o}bius flow \cite{Sarnak},\cite{VeechNotes2},\cite{AV}]\label{SV}
There exists a unique admissible measure $\eta_M$ on $X_{\A_3}$ which is ergodic with the Pinsker $\sigma$-algebra
$$
\Pi_{\eta_M}(S)=s^{-1}\Big(\B(\A_2)\Big)\,.
$$
Moreover, $\eE(\textrm{pr}_1|_{\Pi_{\eta_M}(S)}) = 0$.
\end{thm}

\smallskip

\noindent We shall refer to this measure $\eta_M$ as the `Chowla-Sarnak-Veech measure', (or `CSV measure').
W. Veech in \cite{VeechNotes2}, \cite{AV} addresses this measure as `Chowla measure'. Obviously, under
$\eta_M$, the spectral measure of $\textrm{pr}_1$ is the Lebesgue measure. The following crucial computation
is due to W. Veech \cite{VeechNotes2}.

\smallskip

\begin{prop}\label{VMV} For any invariant measure $\eta \in \mathcal{I}_S(\bmu)$, we have
${\textrm{pr}_1} \in {L^2(X_{\A_3}, \Pi(S),\eta)}^{\perp}$.
\end{prop}
\begin{proof} Let $f \in C(X_{\A_2})$ and put $F = f \circ s$. Let $\eta \in \mathcal{I}_S(\bmu)$. Then,
\begin{align}
\int {\textrm{pr}_1}(x)F(x) d\eta (x) & = \lim_{k \longrightarrow +\infty}
\frac{1}{N_k}\sum_{n=1}^{N_k}{\textrm{pr}_1}(S^n\bmu)F(S^n\bmu)\\
&=\lim_{k \longrightarrow +\infty} \frac{1}{N_k}\sum_{n=1}^{N_k}\bmu(n) f(S^n\bmu^2)\\
&=0.
\end{align}
But the space $\mathcal{S}=\big\{f \circ s, f \in C(X_{\A_2})\big\}$ is dense in 
$L^2(X_{\A_3},s^{-1}(\mathcal{B}(\A_3)),\eta)$. Therefore, \linebreak ${\textrm{pr}_1} \in {L^2(X_{A_3}, \Pi(S),\eta)}^{\perp}$
and the proof is complete. 
\end{proof}

\bigskip
\smallskip

\noindent{}The following is a crucial consequence of Veech's computation.

\smallskip
\smallskip

\begin{prop}
For the M\"{o}bius flow \((X_{\mathcal{A}_3}, \mathcal{B}(\mathcal{A}_3), S)\), we have:
%\begin{enumerate}[label=(\alph*)]
 %   \item \label{V1} 
 The entropy of the measure \(\eta_M\) is strictly positive.
  %  \item \label{V2} Every potential measure \(\eta \in \mathcal{I}_S(\mu)\) has positive entropy.
%\end{enumerate}
\end{prop} 
 \begin{proof}%We start by proving \ref{V1}. 
 Assume by contradiction that $\eta_M$ has a zero entropy. Then
$\Pi_{\eta_M}(S)=\mathcal{B}(\A_3)$ mod $\eta_M$, where $\mathcal{B}(\mathcal{A}_3)$ is the whole $\sigma$-algebra. We thus have that
$$
\eE(\textrm{pr}_1|_{\Pi_{\eta_M}(S)}) = \textrm{pr}_1=0\,,
$$
which is impossible since $\eta_M\Big(\{x \in X_{\A_3}: \textrm{pr}_1(x)=\pm 1\}\Big)>0.$

%\noindent For \ref{V2}. Let $\hat E : L^2(X_{\A_3}, \B(\A_3),\eta) \to L^2(X_{\A_3},\Pi_\eta(S),\eta)$
%be the projection operator given by the conditional expectation with respect to the $\sigma$-algebra $\Pi_\eta(S)$.
%By the above computation $\eE\Big(\hat E({\textrm{pr}_1})|_{\Pi(S)}\Big) = 0$. In fact the same computation also shows that in fact
%$\eE\Big(\hat E({\textrm{pr}_k})|_{\Pi(S)}\Big) = 0$ for all $k\in \N$. Since the linear span of the family $\{\textrm{pr}_k\,,k\in \N\}$
%is dense in $L^2(X_{\A_3}, \B(\A_3),\eta)$, it follows that for any $f \in L^2(X_{\A_3}, \B(\A_3),\eta)$, $\eE\Big(\hat E (f))|_{\Pi(S)}\Big)=0$, that, $\hat E' \circ \hat E$ is the zero operator, where $\hat E'$ is the  projection operator on $L^2(X_{\A_3},\Pi(S),\eta)$.

%On  the other hand if $\eta$ has zero entropy, $\Pi_{\eta}(S)$ is the trivial sigma algebra and hence
%$\hat E$ is the identity operator. Thus for any set $C \in \B(\A_3)$, $(\hat E' \circ E)(\1_C) =\hat E'(\1_C) = 0$.
%Thus for any $A \in \Pi(S)$, $\eta (C \cap A) = 0$,  which implies that the restriction of $\eta$ to $\Pi(S)$is the zero measure. This contradicts the fact
%that it projects onto the non-zero measure $\nu_M$.
\end{proof} 

\smallskip

\noindent {\bf STEP III : A consequence of Rokhlin-Sinai's Theorem} :

\smallskip

Now we shall use the arguments of Rokhlin-Sinai (Theorem \ref{Q} combined with Proposition \ref{P} ) to prove the following.

\smallskip

\begin{prop} \label {abscont} For all potential measures $\eta \in \mathcal{I}_S(\bmu)$
, the spectral measure of every $f\in {L^2(X_{\A_3},\Pi(S),\eta)}^{\perp}$ is absolutely continuous with
respect to Lebesgue measure.
\end{prop}
\begin{proof} We begin by fixing any invariant measure $\eta \in \mathcal{I}_S(\bmu)$. We
recall that,

\smallskip

\noindent (i) $\Pi(S)=s^{-1}(\mathcal{B}(\A_2))$ is the Pinsker $\sigma$-algebra of $\eta_M$,

\smallskip

\noindent (ii) $\Pi_{\eta}(S) \supset \Pi(S)$, where $\Pi_{\eta}(S)$ is the Pinsker $\sigma$-algebra of $\eta$.

\smallskip

\noindent (iii) $\eE({\textrm{pr}_1}|_{\Pi_{\eta}(S)}|_{\Pi(S)}) = \eE({\textrm{pr}_1}|_{\Pi(S)}) =  0$.

\noindent Now note that $L^2(X_{\A_3}, \B(\A_3),\eta) = L^2(X_{\A_3},\Pi(S),\eta)\oplus L^2(X_{\A_3},\Pi(S),\eta)^{\perp}$
and since $\Pi(S) \subset \Pi_{\eta}(S)$, we can write
$L^2(X_{\A_3},\Pi(S),\eta)^{\perp} = L^2(X_{\A_3},\Pi_{\eta}(S),\eta)^{\perp}\oplus V$, where
$$
V = L^2(X_{\A_3},\Pi_{\eta}(S),\eta)\circleddash L^2(X_{\A_3},\Pi(S),\eta) \equiv
\{f \in  L^2(X_{\A_3},\Pi_\eta(S),\eta), f \perp L^2(X_{\A_3},\Pi(S),\eta) \}\,.
$$
Thus,
$$
L^2(X_{\A_3}, \B(\A_3),\eta) = L^2(X_{\A_3},\Pi(S),\eta)\oplus L^2(X_{\A_3},\Pi_{\eta}(S),\eta)^{\perp}\oplus V\,.
$$
We are going to see that the spectral type of $U_S$ on %the second and the third factor of 
the above decomposition is Lebesgue and this will be a consequence of positivity of entropy of $\eta_M$.

%Since entropy of $\eta$ is positive, it follows that the unitary operator $U_S$ acting on the second factor $L^2(X_{\A_3},\Pi_\eta(S),\eta)^\perp$ has countable Lebesgue spectrum by applying Rokhlin-Sinai's Proposition (\ref {P}).

Indeed, we start by establishing the spectral type of $U_S$ on the third factor $V$ is countable Lebesgue, for that we need to `dig' a bit deeper. This means we need to use the Rokhlin-Sinai's Theorem (\ref {Q}) in the construction of a basis for $V$ that gives a countable Lebesgue spectrum. This is done as follows. Since the entropy of \( \eta_M \) is positive, applying the Rokhlin–Sinai Theorem (\ref{Q}) to the dynamical system $(X_{\A_3},\B(\A_3),\eta_M)$ provides a sub-\(\sigma\)-algebra \( \mathcal{E} \) such that 
\[
S^{-1}\mathcal{E} \subset \mathcal{E}, \qquad 
S^n \mathcal{E} \nearrow \mathcal{B}(\mathbb{A}_3), 
\quad \text{(i.e., } \bigcup_n S^n \mathcal{E} \text{ generates } \mathcal{B}(\mathbb{A}_3)), \quad \text{and}\ 
S^{-n}\mathcal{E} \searrow \Pi(S).
\]

Now apply Theorem \ref{CFS-F} by taking the algebra $\F$ to be $\F = \mathcal{E}\cap\Pi_\eta(S)$. Note that the hypothesis of this theorem holds and the Hilbert space $\ds \bigoplus_{n \in \Z}H_n$ of that theorem is just $L^2(X_{\mathbb{A}_3}, \Pi_\eta(S), \eta)\ominus L^2(X_{\mathbb{A}_3},\cap \Pi(S), \eta) = V$, (this follows from the fact that $S^n \mathcal{E}\cap \Pi_\eta (S) \nearrow \mathcal{B}(\mathbb{A}_3)\cap \Pi_\eta (S) = \Pi_\eta (S)$, and $S^{-n}\mathcal{E} \cap \Pi_\eta (S)\searrow \Pi(S)\cap \Pi_\eta(S) = \Pi(S)$. In other words if
$$
W = L^2(X_{\mathbb{A}_3}, \mathcal{E} \cap \Pi_\eta(S), \eta)
\ominus L^2(X_{\mathbb{A}_3}, S^{-1}\mathcal{E} \cap \Pi(S), \eta)\,.
$$
Then $W$ is infinite-dimensional and by taking $\{f_j, j \in J\}$ as an orthonormal basis of $W$, we see that  $\{U_S^n f_j, j \in J, n \in \Z\}$ is an orthonormal basis of $V = L^2(X_{\A_3},\Pi_\eta(S),\eta)\circleddash L^2(X_{\A_3},\Pi(S),\eta)$, (by separability, $J$ is countable). Thus, Theorem \ref{CFS-F} shows that restriction of $U_S$ to $V$ also has countable Lebesgue spectrum.\\

To finish the proof, we notice that either the entropy $h_\eta(S)$ of $\eta$ is zero or strictly positive. If $h_\eta(S)=0$, then $ L^2(X_{\A_3},\Pi_{\eta}(S),\eta)^{\perp}=\Big\{0\Big\}$ and 
$$L^2(X_{\A_3}, \B(\A_3),\eta) = L^2(X_{\A_3},\Pi(S),\eta)\oplus L^2(X_{\A_3},\Pi(S),\eta)^{\perp}.
$$
Therefore, our previous proof shows that  $U_S$ on $L^2(X_{\A_3},\Pi(S),\eta)^{\perp}$ has countable Lebesgue spectrum. Otherwise, the entropy $h_\eta(S)$  is positive and in this case it follows that the unitary operator $U_S$ acting on the second factor $L^2(X_{\A_3},\Pi_\eta(S),\eta)^\perp$ has countable Lebesgue spectrum by applying Rokhlin-Sinai's Proposition \ref {P}. Let us further notice that the spectrum of $U_S$ on  $L^2(X_{\A_3},\Pi(S),\eta)$ is discrete.
\smallskip

\noindent {\bf STEP IV : Using the notion of `affinity' to conclude the proof of Theorem \ref{main}}

\smallskip

It follows from the previous analysis that the spectral type of ${\textrm{pr}_1}$ is absolutely continuous with respect to Lebesgue measure.

Finally, by applying Lemma \ref{BL} combined with Remark \ref{IG}, we get that for any $x \in X$,
$$
\limsup \Big|\frac{1}{N}\sum_{n=1}^{N}\bmu(n)f(T^nx)\Big|
\leq \sup G(\sigma_{{\textrm{pr}_1},\eta},\sigma_{f,\nu_x})\,,
$$
where $\nu_x \in \mathcal{I}_T(x)$, and $\mathcal{I}_T(x)$ is the weak-star closure set of
$\Big\{\frac{1}{N}\sum_{n=1}^{N}\delta_{T^nx}\Big\}$. We thus get, by our assumption,
that the right-side is zero, and this finishes the proof. %{\hfill{$\Box$}}
\end{proof}

\smallskip

%Moreover, by Lemma \ref{BL} and Remark \ref{IG}, we have for any $x \in X$,
%$$
%\limsup \Big|\frac{1}{N}\sum_{n=1}^{N}\bmu(n)f(T^nx)\Big|
%\leq \sup G(\sigma_{{\textrm{pr}_1},\eta},\sigma_{f,\nu_x})\,
%$$
%where $\nu_x \in \mathcal{I}(x)$, and $\mathcal{I}(x)$ is the weak-star closure set of
%$\Big\{\frac{1}{N}\sum_{n=1}^{N}\delta_{T^nx}\Big\}$. Now, applying Lemma \ref{P}, we get %the right-side is zero, and this
%finishes the proof.

\medskip

\section{Some consequences of the main result}

\subsection{Applications to Number theory}
\medskip

Now we state a few corollaries, the first one follows rather immediately.

\smallskip

\begin{cor}\label{rigid}The M\"{o}bius disjointness conjecture holds for any rigid transformation. 
\end{cor}

\smallskip

\noindent We recall that the transformation $T$ is rigid if there is a sequence of integers $(n_k)$ such that
$(T^{n_k})$ converge strongly to the identity map. As a consequence, we obtain an improvement of the
Theorem 5.1 in \cite{FW} which extended substantially the main result in \cite{KLR}. In \cite{FW},
the author proved that the M\"{o}bius orthogonality holds for the rigid map under an extra-condition.
In fact, our proof allows us to obtain as a corollary the main ingredient needed in his proof. This
follows from our Corollary \ref{FWLG}.\\

% According to Halaz criterion this condition insure that a system cannot be weak-mixing. 

\noindent The second corollary follows from the observation in Proposition~\ref{abscont}. This provides a partial answer to the question posed in~\cite{AD} and has a direct application in number theory. It is stated as follows.
\smallskip

\begin{cor}\label{potentiel-Lebesgue}All the potential spectral measures of M\"{o}bius and Liouville
function are absolutely continuous with respect to the Lebesgue measure. 
\end{cor}

\smallskip

\begin{rem} As a corollary to our Theorem (\ref{main}), Sarnak's M\"{o}bius disjointness is fulfilled for
a system for which every invariant measure has discrete spectrum. An earlier approach to this result
was developed by applying Matom\"{a}ki-Radziwi{\l}{\l}-Tao's Theorem  on the average Chowla
conjecture, (see \cite{MRT}). More precisely, the crucial argument in that approach was based
on the fact that the average of the correlations of order two of M\"{o}bius and Liouville functions
converge to 0 with logarithmic speed (see \cite[Proposition 5.1]{MRT}). Many authors used this
approach to establish the M\"{o}bius disjointness conjecture for such systems, \cite{ALR}, \cite{HWZ} \cite{HZY}.
Our method bypasses all of this and obtains the validity of Sarnak conjecture for a much wider
class of systems.

We established that `Veech Systems' have the property that all invariant measures have
discrete spectrum (Theorem 4.20 in \cite{AM}), and in particular the translation flow on the orbit
closure of a `Veech function' on $\Z$ has the same property (see \cite{AM} for the definitions).
Hence the translation flow on the orbit closure of the Veech function satisfies the M\"{o}bius
randomness law. This along with results in \cite{AM} allows us to obtain a dynamical proof of
Motohashi-Ramachandra type theorem and Matom\"{a}ki-Radziwi{\l}{\l}'s result \cite{MR}
on the short interval for the case of Liouville and M\"{o}bius function. This latter result is stated as Corollary \ref{MRshort}
% can be
%stated as follows :

%\noindent For any $\epsilon>0$, for any $H \leq X$ large, we have 
%$$
%\int_{X}^{2X}\Big|\sum_{x < k<x+H}\bmu(k)\Big| dx=o(HX)\,.
%$$

\end{rem}

\smallskip

\noindent Corollary \ref{potentiel-Lebesgue} allows us also to obtain as a consequence the main theorem in
\cite{MRT} and Corollary 1.2 from \cite{MR}. Precisely, we have the following.

\smallskip

\begin{cor}\label{main-MRT}\ \\
(I) [Matom\"{a}ki-Radziwi{\l}{\l}-Tao Theorem \cite{MRT}] For any $\epsilon , \delta \in (0,1)$, 
There is a large  but  fixed $H = H(\delta)$ such  that, for all large enough $X$, the cardinality
of the set of all $h$'s with $|h| \leq H$ satisfying
\begin{align}\label{Wmixing}
\Big|\sum_{1 \leq j \leq X}\bml(j) \bml(j+h)\Big| \leq \delta X\,.
\end{align}
is greater than $(1-\eps )H$.

\noindent (II) [Matom\"{a}ki-Radziwi{\l}{\l} Corollary \cite{MR}] We further have, for any $h \geq 1$,
there exists $\epsilon(h)>0$ such that 
\begin{align}\label{ortho-D}
\Big|\frac{1}{X}\sum_{1 \leq j \leq X}\bml(j) \bml(j+h)\Big| \leq 1-\epsilon(h).
\end{align}
for all large enough $X$. 
\end{cor}
\begin{proof} (I) Let $(X_k)$ be a subsequence such that $X_k \longrightarrow +\infty$. Then,
$$
\Big|\sum_{1 \leq j \leq X_k}\bml(j) \bml(j+h)\Big| \tend{k}{+\infty}{}\int z^{-h} f(z) dz\,,
$$
for some function $f \in {L}^{1}(dz)$. But we also have 
$$
\frac{1}{H}\sum_{h=1}^{H}|\widehat{f}(h)| \tend{H}{\infty}{}0\,.
$$
This proves the first part.

\noindent (II) For the second part, assume by contradiction, that there is $h \geq 1$ such that
for any $\epsilon>0$, we have 
$$
\limsup_{X \longrightarrow +\infty}\Big|\frac{1}{X}\sum_{1 \leq j \leq X}\bml(j) \bml(j+h)\Big| \geq 1-\epsilon \,.
$$
Then, by taking a subsequence which may depend on $h$ , we get
$$
\lim_{k \longrightarrow +\infty}\Big|\frac{1}{X_k}\sum_{1 \leq j \leq X_k}\bml(j) \bml(j+h)\Big|=
\Big|\widehat{f}(h)\Big|=1\,,
$$
which is impossible by Cauchy-Schwarz inequality.
\end{proof}
%More details for us only, write 
%\begin{align*}
 % \Big|\widehat{f}(h)\Big|&=\int z^{-h}\sqrt{f(z)} \sqrt{f(z)} dz 
 % &\leq (\int |z^{-h}\sqrt{f(z)}|^2 dz)^{\frac12} (\int |sqrt{f(z)}|^2 dz)^{\frac12}=1.
%\end{align*}
%We thus have equality in Cauchy-Schwartz inequality. We conclude by considering a set of positive measure
\begin{cor}\label{FWLG} For any $k \in \N$ and for a large $h$,
$\ds \limsup_{N \longrightarrow +\infty}\frac{1}{N}\sum_{n=1}^{N}\Big| \sum_{l=1}^{h}\bmu(n+kl)\Big|^2=o(h^2).$
\end{cor} 
\begin{proof}Let $\sigma_{\bmu}$ be a potential spectral measure. Then 

$$
\frac{1}{N_j}\sum_{n=1}^{N_j}\Big| \sum_{l=1}^{h}\bmu(n+kl)\Big|^2 \tend{j}{+\infty}{}\int_{\T}\Big|D_{h}(k\theta)\Big|^2 
d\sigma_{\bmu}\,,
$$
where $D_{h}(\theta)=\ds \sum_{l=1}^{h}e^{i  l \theta}$ is the classical Dirichlet kernel. We thus get, by Corollary
\ref{potentiel-Lebesgue} ,
$$
\int_{\T}\Big|D_{h}(k\theta)\Big|^2 d\sigma_{\bmu}=\int \Big|D_{h}(k\theta)\Big|^2 f(\theta) d\theta\,,
$$
for some function $f \in L^1(dz)$. Furthermore, the functions $\phi_h(\theta)=\frac{1}{h}D_{h}(\theta)$
tend to zero uniformly outside any neighborhood of $0$.   The rest of the proof is left to the reader.
\end{proof}

\smallskip

\begin{rem} Let us point out that our Corollary \ref{potentiel-Lebesgue} assert much more than
\eqref{Wmixing} and \eqref{ortho-D}. Indeed, it states that any potential spectral measures of
M\"{o}bius function $\sigma_{\bmu}$ or Liouville function $\sigma_{\bml}$ is a Rajchman measure,
that is, $\big|\widehat{\sigma_{\bmu}}(n)\big|, \big|\widehat{\sigma_{\bml}}(n)\big| \tend{|n|}{+\infty}{}0$.
However, \eqref{Wmixing} assert only that  $\sigma_{\bml}$ is a continuous measure and \eqref{ortho-D}
that $\sigma_{\bml}$ is in the orthogonal to Dirichlet measures.

More precisely, let  $\M(\T))$ be the algebra of  the regular Borel complex measures on the torus
$\T$, equipped with the convolution product of measures, defined by $\mu*\nu$. This is the pushforward
measure of $\mu \otimes \nu$ under the map $a : (x,y) \in T \times T \mapsto x+y \in \T$.
%Denote by $\Lambda$ the character group of $\T$. 
We shall call a subset $L\subset \M(\T)$ a $L$-subspace (resp. $L$-ideal or  $L$-sub-algebra) of
$\M(\T)$ when $L$ is a closed subspace (resp. ideal or  sub-algebra) of $\M(\T)$ that is invariant
under absolute continuity of measures. This means
$$
\text {if} \; \; \mu \in L  \; \; \textrm{and} \; \; \nu \ll \mu \; \textrm{then}\;\; \nu \in L.
$$
It is well know that the sets $\M_c(T) \setdef\Big\{\mu \in \M(\T)\ |\; \mu \;\textrm{is a continuous
measure}\Big\}$ and \linebreak $\M_0(T)\setdef\Big\{\mu \in \M(\T)\ |\; \mu \;\textrm{is a Rajchman
measure}\Big\}$ are $L$-ideals and
$$
\mu \in \M_c(T)  \Longleftrightarrow \lim_{N \longrightarrow +\infty} \frac{1}{N}\sum_{n=0}^{N}\big|\widehat{\sigma_{\mu}}(n)\big|^2 \tend{N}{+\infty}{}0\,.
$$	
Given an $L$-ideal, we define its orthogonal as follows
$$
L^{\perp}=\Big\{\mu \in \M(\T):~~|\mu|\perp |\nu ,|~~\forall \nu \in L\Big\}\,.
$$
Furthermore the following decomposition is well know.
$$
\M(T)=L \oplus L^{\perp}\,.
$$
The probability measure $\mu$ on $\T$ is said to be a Dirichlet measure if 
$$
\limsup_{\gamma \longrightarrow +\infty}|\widehat{\mu}(\gamma)|=1\,.
$$
It is well known  that the set $\mathcal{D}(\T) \setdef \Big\{\mu \ |\ \text {is a Dirichlet measure}\Big\}$
is an $L$-ideal. Moreover, 
$$ 
\mu \in \mathcal{D}(\T)^{\perp} \Longleftrightarrow \limsup_{k \longrightarrow +\infty}|\widehat{\mu}(n)|<1\,.
$$
and 
$$
M_0(\T) \subset \mathcal{D}(\T)^{\perp} \subset \M_c(\T)\,.
$$
For more details and proofs, we refer to \cite[Chap. II]{HPM}.
\end{rem}

\smallskip

\begin{rem} We established that `Veech Systems' have the property that all invariant measures have
discrete spectrum, (Theorem 4.20 in \cite{AM}). Hence, (without using the earlier result for such systems),
our main theorem directly implies that the translation flow on the orbit closure of a `Veech function' on $\Z$
satisfies the Sarnak's conjecture (see \cite{AM} for the definitions). This along with results in \cite{AM}
allows us to obtain a dynamical proof of Motohashi-Ramachandra type theorem and Matom\"{a}ki-Radziwi{\l}{\l}'s
result \cite{MR} on the short interval for the case of Liouville and M\"{o}bius function. This latter result can be
stated as follows :

\smallskip

\begin{cor}\label{MRshort} For any $\epsilon>0$, for any $H \leq X$ large, we have 
$$
\int_{X}^{2X}\Big|\sum_{x < k<x+H}\bmu(k)\Big| dx=o(HX)\,.
$$
\end{cor}
\end{rem}

\smallskip

\noindent {\bf Other consequences of Theorem (\ref{SNMV})}

\smallskip

In this subsection we derive some M\"{o}bius disjointness type results for special dynamical systems
as consequences of Theorem (\ref{SNMV}).  The following corollary of Theorem \ref{SNMV}
follows from the spectral theorem.

\smallskip

\begin{cor}Let $(X,T,\mu)$ be a compact metric, topological dynamical system. Then for any
\linebreak $F \in L^2(X,\mu)$, for any $f \in C(X_{\A_2})$, we have
$$
\frac{1}{N}\sum_{n=1}^{N}\bmu(n)f(\bmu^2(n))F(T^nx) \tend{N}{+\infty}{L^2}0\,.
$$
\end{cor}
\begin{proof}By the spectral theorem, we can write
\begin{eqnarray*}
	\Big\|\frac{1}{N}\sum_{n=1}^{N}\bmu(n)f(\bmu^2(n))F(T^nx)\Big\|_{
		L^2(X,\mu)}=\Big\|\frac1{N}\sum_{n=1}^{N}\bmu(n)f(\bmu^2(n))e^{in \theta}\Big\|_{L^2(\sigma_F)},\\,
\end{eqnarray*}
where $\sigma_F$ is the spectral measure of $F$ for the Koopman operator $F \mapsto F \circ T$. Now,
by Theorem \ref{SNMV}, we have 
$$
\frac{1}{N}\sum_{n=1}^{N}\bmu(n)f(\bmu^2(n)) e^{in\theta} \tend{N}{+\infty}{}0\,.
$$
Moreover, the sequence $\Big(\frac{1}{N}\sum_{n=1}^{N}\bmu(n)f(\bmu^2(n)) e^{in\theta}\Big)_{N \in \N}$
is bounded. Hence, by the Lebesgue dominated  convergence  theorem, we obtain 
$$
\Big\|\frac{1}{N}\sum_{n=1}^{N}\bmu(n)f(\bmu^2(n))F(T^nx)\Big\|_{
L^2(X,\mu)}  \tend{N}{+\infty}{}0\,.
$$
The proof of the corollary is complete.
\end{proof}\\

\smallskip

At this point we ask whether in the above corollary the convergence can be almost sure convergence?
In the class of dynamical systems for which every invariant measure has discrete spectrum or the
Lebesgue spectrum, this can be done. 

\smallskip

\begin{prop}
Let $(X,T)$ be a dynamical system for which every invariant measure has discrete spectrum. Then
for any $f\in C(X_{\A_2})$ and $F\in C(X)$ and for any invariant measure $\mu$,
$$
\lim\limits_{N\to \infty}\frac{1}{N}\sum_{n=1}^{N}\bmu(n)f(\bmu^2(n))F(T^nx) = 0\,,\ \mu\ \textrm {a.e.}\,.
$$
In particular, for $k \geq 2$ and $a_1, \cdots,a_k$ be distinct non-negative integers, for any $F\in C(X)$,
on a set of points $x$ of full measure we have
$$
\lim\limits_{N\to \infty}\frac {1}{N}\sum_{n \leq N}\bmu(n)\bmu^2(n+a_1)\cdots\bmu^2(n+a_k)F(T^n x) = 0\,.
$$
\end{prop}
\begin{proof}  Recall that the set of points $x$ which are generic for some ergodic invariant measure have
full measure, (see \cite[Proposition 5.12]{DGS}). By our hypothesis, such $x$ is generic for an ergodic measure
with discrete spectrum. Hence the map $n\mapsto F(T^n x)$ is  Besicovitch almost periodic.
\end{proof}

\smallskip

\begin{rem} We remark that if for some invariant measure $\mu$, the system $(X,T,\mu)$ has Lebesgue
spectrum, then the above convergence also hold for $ F\in L^1(X,\mu )$, $\mu$ a.e.$x$. This follows
from, \cite[Theorem 4.2]{BeL}. 
\end{rem}

\smallskip

For certain very special class of dynamical systems this almost sure convergence can be extended to
convergence everywhere.

\smallskip

\begin{prop} Let $(X,T,\mu )$ be mean equicontinuous compact metric dynamical system. Then the
above convergence holds for any $x\in X$.
\end{prop}
\begin{proof} Since $(X,T,\mu)$ is mean equicontinuous, the orbit closure of each $x\in X$ is
uniquely ergodic with discrete spectrum, Hence the map $n\to F(T^n x)$ is  Besicovitch almost periodic
for any \linebreak $F\in C(X)$.
\end{proof}\\

\smallskip

\noindent We thus ask the following question:

\smallskip

\begin{que}\label{question} Do we have for any dynamical system $(X,T,\mu )$ with topological
entropy zero, for any $f\in C(X_{\A_2})$ and $F\in C(X)$, for any $x \in X$,
$$
\lim\limits_{N\to \infty}\frac {1}{N}\sum_{n \leq N}\bmu(n)f(\bmu^2(n))F(T^n x) = 0\,?
$$  
\end{que} 

\medskip

We ask also on the convergence almost everywhere for any dynamical system. Let us further
notice that Question \ref{question} is a weak form of the strong Sarnak M\"{o}bius conjecture.

%We point out that according to Murty-Vatwani's theorem \cite[Theorem 1]{MV} (see also Theorem 2. where
%the result is obtained under GRH), it can be proved that for any dynamical system $(X,T,\mu )$,
%for any  $F\in L^1(\mu)$, for almost all $x \in X$, 

%$$\lim\limits_{N\to \infty}\frac {1}{N}\sum_{n \leq N}\bmu(n)\bmu^2(n+a_1)\cdots\bmu^2(n+a_k)F(T^n x) = 0.$$

%Moreover, by Lemma \ref{BL} and Remark \ref{IG}, we have for any $x \in X$,
%$$
%\limsup \Big|\frac{1}{N}\sum_{n=1}^{N}\bmu(n)f(T^nx)\Big|
%\leq \sup G(\sigma_{{\textrm{pr}_1},\eta},\sigma_{f,\nu_x})\,
%$$
%where $\nu_x \in \mathcal{I}(x)$, and $\mathcal{I}(x)$ is the weak-star closure set of
%$\Big\{\frac{1}{N}\sum_{n=1}^{N}\delta_{T^nx}\Big\}$. Now, applying Lemma \ref{P},
%we get the right-side is zero, and this %finishes the proof.

\medskip
\subsection{A M\"obius-weighted ergodic theorem for weakly rigid contractions}\label{Lin}
\medskip

\subsubsection{M\"obius-weighted ergodic theorems}
	
%Let $\mu: \mathbb N^+ \longrightarrow \{-1,0,1\}$ be the M\"obius function, defined by
%$\mu(1)=1$, $\mu(n)=(-1)^n$ if $n>1$ is the product of $n$ distinct primes, and 
%$\mu(n)=0$ if $n$ has a square factor.

%P. Sarnak conjectured that if $\tau$ is a continuous map of a compact metric space $X$ into 
%itself with zero topological entropy, then $\frac1N\sum_{n=1}^N \mu(n)f(\tau^n x) \to 0$
%for every $f \in C(X)$ and $x \in X$; equivalently, the operator $Tf =f \circ \tau$ on 
%$C(X)$ satisfies $\frac1N\sum_{n=1}^N \mu(n) T^nf \to 0$ weakly, for every $f \in C(X)$.
%\medskip

In \cite{AL} it is proved that any weakly almost periodic operator $T$ on a Banach space $E$
satisfies 
\begin{equation} \label{strong-mobius}
\Big\|\frac1N\sum_{n=1}^N \mu(n)T^n v\Big\| \to 0, \quad \forall v\in E.
\end{equation}
 \eqref{strong-mobius} was also proved for other classes of operators. In connection with 
Sarnak's conjecture, a topological entropy $h^*_{top}(T)$ for power-bounded operators was 
defined, and it was proved that {\it if Sarnak's conjecture is true}, and $T$ is power-bounded 
with $h^*_{top}(T)=0$, then 
\begin{equation} \label{mobius}
\frac1N\sum_{n=1}^N \mu(n)T^n v \to 0 \quad \text{weakly}, \quad \forall v\in E.
\end{equation}
\medskip

{\bf Definition.} An operator $T$ on a Banach space $E$ is called {\it (weakly) rigid} if 
for some increasing subsequence $(n_k)$ we have $T^{n_k}v \to v$ (weakly) for every $v \in E$.

It was shown in \cite[Proposition 7.2]{AL} that a power-bounded weakly rigid $T$ is invertible,
with $T^{-1}$ also power-bounded. In \cite[Theorem 3.5]{AL} it was proved that if $T$ is a
weakly rigid contraction on a separable Banach space $E$, then $h^*_{top}(T)=0$. 
As an application of our main result, we show in this section that every weakly rigid 
power-bounded $T$ satisfies \eqref{mobius}.
\medskip

\subsubsection{A M\"obius-weighted ergodic theorem for weakly rigid operators}

In \cite[Corollary 2.12]{AL} it is proved that a weakly rigid power-bounded $T$, on a Banach 
space $E$ which does not contain an isomorphic copy of $\ell_1$, satisfies \eqref{mobius}. 
Here we remove the restriction on $E$.

\begin{thm} \label{rigid-II}
 Let $T$ be a weakly rigid power-bounded operator on a Banach space $E$. Then
$$
\frac1N\sum_{n=1}^N \mu(n)T^n v \to 0 \quad \text{weakly}, \quad \forall v\in E.
$$
\end{thm}
\begin{proof}  The operator $T$ is a contraction in the equivalent norm 
$\||v\||:= \sup_{n\ge0} \|T^nv\|$, so we assume that $T$ is a contraction,
and by \cite[Proposition 7.2]{AL}  it is an invertible isometry.

Fix $v \in E$; by looking at the closed linear manifold generated by the orbit $(T^nv)_{n\ge 0}$,
we may and do assume that $E$ is separable. Hence $B$, the unit ball of $E^*$, 
is compact metric in the weak* topology, and the restriction $\tau:=T^*_{|B}$ is a 
homeomorphism of $B$. We put $Sf:=f\circ \tau$ for $f \in C(B)$.

Since $T$ is weakly rigid, $\tau^{n_k}\phi \to \phi$ weak* for every $\phi \in B$, 
so $S^{n_k}f(\phi) \to f(\phi)$ for every $\phi \in B$ and  $f \in C(B)$. Let $\nu$ be a
$\tau$-invariant probability on the Borel sets of $B$. Then $S^{n_k}f \to f$ in $L^2(B,\nu)$
when $f \in C(B)$, by Lebesgue's dominated convergence theorem. Since $C(B)$ is dense in 
$L^2(B,\nu)$, we have that $S$ is rigid on $L^2(B,\nu)$. By Baxter's \cite[Theorem 1, p. 2.20]{B},
see also \cite[Theorems 4 and 5]{A}, $S$ on $L^2(B,\nu)$ has singular maximal spectral type.
%\footnote{See \cite{N} for the definition and proof of existence of the maximal spectral type}.
We thus obtain that for every $\tau$-invariant probability $\nu$, the maximal spectral type of 
$S$ in $L^2(B,\nu)$ is singular. By el Abdalaoui and Nerurkar (see [Theorem 3.1, in arxiv.org/abs/2006.07646, p.~8])
$$
\frac1N\sum_{n=1}^N f(\tau^n \phi) \to 0 \quad \forall f\in C(B),\  \forall \phi \in B.
$$
For $v \in E$, we take $f(\phi)=\phi(v)$ and obtain \eqref{mobius}.
\end{proof}

{\bf Remarks.} 1. The reduction at the beginning of the proof of Theorem \ref{rigid-II}
shows that we may assume that for each $v \in E$ there is a sequence $(n_k)$,
which may depend on $v$, such that $T^{n_k}v \to v$ weakly.

2. Theorem \ref{rigid-II} improves \cite[Proposition 3.6]{AL}, in which there is a constraint
on $(n_k)$, due to Kanigowski et al., and \cite[Proposition 3.7]{AL}, in which weak
convergence of {\it logarithmic} averages is proved for rigid operators.

3. As shown in \cite[Propositions 7.1 and 7.3]{AL}, weak rigidity does not imply rigidity, 
and rigidity does not imply mean ergodicity.

\medskip

\section{Veech's Conjecture and Proof of Veech's theorem (Theorem \ref{Veech}).}\label{Proof-of-Veech}

\medskip

At this point, we state a conjecture of W. Veech, (which he announced in his unpublished notes
\cite{VeechNotes2}). As we shall see, the proof of this theorem actually allows us to view our main
theorem differently. We shall first give a proof of his theorem and then, (using the notation developed
in the proof), discuss our main result in the light of this. 

\smallskip

\begin{conj} \label{Veech-conj} [W. Veech] For any $\eta \in \mathcal{I}_S(\bmu)$, we have
$$
\textrm{pr}_1 \in L^2(X_{\A_3},\Pi_{\eta}(S),\eta)^{\perp}\,,
$$
\end{conj}

\smallskip

In \cite{Veech-preprint}, W. Veech proved that his conjecture implies Sarnak's M\"{o}bius disjointness. In the following, we state this precisely.

\smallskip

\begin{thm}[Veech's Theorem \cite{VeechNotes2}]\label{Veech} Suppose that for any
$\eta \in \mathcal{I}_S(\bmu)$, we have
$$
\textrm{pr}_1 \in L^2(X_{\A_3},\Pi_{\eta}(S),\eta)^{\perp}\,,
$$
then Sarnak M\"{o}bius disjointness holds. Here $\Pi_{\eta}(S)$ denotes the Pinsker sigma
algebra of the dynamical system $(X_{\A_3},S,\eta)$.
\end{thm}

\smallskip
\noindent As a consequence of the proof of Theorem \ref{Veech}, we establish that Liouville flow is a factor of M\"{o}bius flow.

\smallskip

\begin{rem} The Chowla conjecture is equivalent to saying that $\bmu$ is quasi-generic,
hence, generic for the Chowla-Sarnak-Veech measure $\eta_M$. Therefore, if Chowla conjecture
holds then our assumption in Theorem \ref{Veech} is satisfied and hence Chowla conjecture
implies Sarnak's M\"{o}bius conjecture.
\end{rem}

\smallskip

\noindent We begin by denoting  $X_{\A^{*}_i},$ $i=2,3$
the subset of sequence $x$ for which the support- $\text{supp}(x)$ is infinite and let
$\Omega_j = \{\pm 1\}^{\Z_j}$, with $\Z_0 = \N \cup \{0\}$ and $\Z_1 = \Z$. We
introduce also the following skew product.
\begin{eqnarray}
& \psi_0 : & X_{\A_2}\times \Omega_0\rightarrow X_{\A_2}\times \Omega_0\\
\nonumber & & (x,\omega)\mapsto (Sx,S^{\textrm{pr}_1(x)}(\omega)).
\end{eqnarray}
The natural extension of $\psi_0$ to $X_{\A_2}\times \Omega_0$ is denoted by $\psi_1$.
Let $\alpha : \Z_1\to \Z_0$ be defined by putting $\alpha(\omega)=(\omega_k)_{k=0}^{+\infty}$.
Obviously, $\alpha$ is onto and  $\alpha \circ \psi_1=\psi_0 \circ \alpha.$ We further define
`co-ordinate wise' a map 
\begin{eqnarray}
& \Phi : & X_{\A_2}\times \Omega_0\rightarrow X_{\A_3}\\
\nonumber & & (x,\omega)\mapsto \Phi(x,\omega) : \textrm{pr}_n(\Phi(x,\omega))=\begin{cases}
0 &\textrm{if~~} n \not \in  \text{supp}(x)=\{n_1<n_2<\cdots<n_k<\cdots\}\\
\textrm{pr}_k(\omega) & \textrm{if~~} n=n_k \in \text{supp}(x)
\end{cases}
\end{eqnarray}
Consequently, we have
\begin{align}
\Phi(X_{\A_2}\times \Omega_0)&=X_{\A_3}, \\
&\textrm{and} \nonumber\\
\Phi \circ \psi_0&=S\circ \Phi   \label{phiformula}
\end{align}
$\Phi$ is also onto but not one to one. However, its restriction $\Phi : X_{\A^{*}_2} \times \Omega_0
\rightarrow X_{\A^{*}_3}$ is an onto homeomorphism. With a slight abuse of notation, denoting by
$\Phi^{-1}_*\eta$-the `push forward' of measure $\eta$ under the map $\Phi^{-1}$, we see that
the measure theoretic dynamical systems $(X_{\A^{*}_2} \times \Omega_0,\psi_0,\Phi^{-1}_*\eta)$
and $(X_{\A^{*}_3},S,\eta)$ are measure theoretically isomorphic by the map $\Phi$, where the
underlying sigma algebras are respective Borel sigma algebras and this holds for any
$\eta \in \mathcal{I}_S(\bmu)$. We also note that, since $\bmu^2$ is a generic point for the
Mirsky measure, it follows that (i) $s_*\eta = \nu_M$ for any $\eta \in \mathcal{I}_S(\bmu)$ and
(ii) $\nu_M(X_{\A_2}\backslash X_{\A^{*}_2}) = 0$. Next, notice that the Pinsker sigma algebra
 of the first of the above isomorphic dynamical systems can be written down by the following two
equal expressions mod null sets,
$$
\bigcap_{n=0}^{+\infty}\psi_0^{-n}\Big(\B(\A_2^* \times \Omega_0)\Big) = \Phi^{-1}\Big(\bigcap_{n=0}^{+\infty}S^{-n}\B(\A_3^*)\Big)\,.
$$
At this point, let us observe that the projection $\textrm{pr}_1$ on $X_{\A_3}$ may be represented in
$X_{\mathcal{A}_2}\setminus\{0\} \times \Omega_0$ by a function $\textrm{Pr}_1$ given by
$$
\textrm{Pr}_1(x,\omega)=\textrm{pr}_1(x)\textrm{pr}_1(\omega)\,.
$$
Moreover, we have
$$
\textrm{Pr}_1(\Phi(x,\omega))=\textrm{pr}_1(x.\omega)\,.
$$
Indeed, by definition, we have 
\[\textrm{Pr}_1(\Phi(x,\omega))=\begin{cases}
	0 &\textrm{if~~}  \textrm{pr}_1(x)=0\\
\textrm{pr}_1(\omega)	 & \textrm{if~~} \textrm{pr}_1(x)=1
\end{cases}\]
and 
\[\textrm{Pr}_1(x.\omega)=\begin{cases}
0 &\textrm{if~~} \textrm{pr}_1(x)=0\\
\textrm{pr}_1(\omega) & \textrm{if~~} \textrm{pr}_1(x)=1.
\end{cases}\]
We thus get 
\begin{eqnarray}\label{lespr}
\forall (x,\omega) \in X_{\A_2} \times \Omega_0,~~~~~ \textrm{Pr}_1 \circ \Phi(x,\omega)= \textrm{pr}_1(x,\omega).
\end{eqnarray}
Let $(\bmu^2,\omega_0)$ be the unique point such that 
\begin{eqnarray}\label{invmu}
\Phi((\bmu^2,\omega_0))=\bmu.
\end{eqnarray}
Therefore
$$
\mathcal{I}_{\psi_{0}}(\bmu^2,\omega_0)=\Phi^{-1}(I_{S}(\bmu))\,.
$$

At this point, notice that for the natural extension $\psi_1$ we have, by Cellarosi-Sinai Theorem \cite{CS}, the system $(X_{\mathcal{A}},S,\nu_M)$ is metrically isomorphic  to the rotation on the compact group $\ds \prod_{p \in \pr}\Z/p^2\Z$. Whence, the entropy with respect to the Mirsky measure $\nu_M$ of $S$ is $0$ and hence  the map $S$ is $\nu_M $.a.e  invertible. Let $\A^*_{2,0}\subset \A^*_2$ be a Borel set such that $\nu_M(\A^*_{2,0})=1$. Then,  $\psi_1$ is an homeomorphism  of $\A^*_{2,0} \times \Omega_1.$ We choose also $\omega_1 \in \Omega_1$ such that $\textrm{pr}_n(\omega_1)=\textrm{pr}_n(\omega_0),$  for all $n \in \N\cup \{0\}$. In the dynamical system $(X_{\A^*_2} \times \Omega_0,\psi_1)$, we keep the notation $\mathcal{I}_{\psi_1}(\bmu,\omega_1)$ for the set of invariant measure that arise from the forward orbit under $\psi_1$ of $(\bmu^2,\omega_1)$. For each $\lambda \in \mathcal{I}_{\psi_1}(\bmu,\omega_1)$, again, by Mirsky-Sarnak theorem, $\textrm{pr}_1\lambda=\nu_M$.  We further have
\begin{claim} 	
The map $\alpha : \lambda \in \mathcal{I}_{\psi_1}(\bmu^2,\omega_1) \longrightarrow \alpha(\lambda) \in \mathcal{I}_{\psi_0}(\bmu^2,\omega_0)$ is an isomorphism onto.
\end{claim}
\begin{proof} Let $\lambda \in \mathcal{I}_{\psi_1}(\bmu,\omega_1)$. Then, there exists a sequence $(N_k)$ such that, for any continuous function $f$ on $\A^*_{2,0} \times \Omega_1$, we have \
$$
\frac{1}{N_k}\sum_{n=0}^{N_k}F(\psi^n(\bmu^2,\omega_1)) \tend{k}{+\infty}{}\int_{\A^*_{2,0} \times \Omega_0}f(x,\omega) d\lambda(x,\omega)\,.
$$
But, because only the forward orbit is involved, the corresponding limit for $(\bmu^2,\omega_1)$ not only exists but does not depend upon the choice of extension of the sequence $\omega_0$ to be a bisequence $\omega_1$. Therefore, $\alpha$ is both onto and invertible.
\end{proof}\\

Now, we start the proof of Theorem \ref{Veech}. Let $T$ be a homeomorphism of a compact metric space $Y$ and assume that $y \in Y$ is completely deterministic, that is, for any $\kappa \in \mathcal{I}_{T}(y)$, the measure-theoretic entropy $h_{\kappa}(T)=0$. Observe that $T$ induces an affine homeomorphism of $\mathcal{P}(Y)$ the space of probability measures on $Y$ equipped with the weak-star topology and $\B(\mathcal{P}(Y))$ the corresponding Borel field. Let us recall also the definition of quasi-factor needed in the proof.

\smallskip

\begin{defn}The dynamical system $(\mathcal{P}(Y)),\B(\mathcal{P}(Y)),\Theta ,T)$ is  said to be a quasi-factor of $(Y,\B(Y),\kappa,T)$ if $\Theta \in \mathcal{P}(\mathcal{P}(Y))$ is
\begin{enumerate}[(i)]
\item $T$-invariant, and
\item for any continuous function $f \in C(Y)$, we have
$$
\int_{\mathcal{P}(Y)}\int_{Y} f(z) d\nu(z) d\Theta(\nu)=\int_{Y}f(z) d\kappa(z)\,,
$$
that is, 
$$
\int_{\mathcal{P}(Y)}\nu d\Theta(\nu)=\kappa\,,
$$
we say that $\kappa$ is the barycenter of $\Theta$.
\end{enumerate}
\end{defn} 
We need also the following theorem due to Glasner and Weiss \cite[Theorem 18.17, p.326]{G},
\begin{lem}\label{GW}The quasi-factor of zero entropy system is a zero entropy system.
\end{lem}

\smallskip

\paragraph{Proof of Theorem \ref{Veech}.} Let $f$ be a continuous function on $Y$ and write

\begin{eqnarray}\label{Jformula1}
\frac{1}{N}\sum_{n=1}^{N}\bmu(n)f(T^ny)= \frac{1}{N}\sum_{n=1}^{N}\textrm{pr}_1(S^n\bmu)f(T^ny).
\end{eqnarray}
  
\noindent{}Taking into account \eqref{lespr} combined with \eqref{invmu}, we can rewrite \eqref{Jformula1} as follows
\begin{eqnarray}\label{Jformula2}
\frac{1}{N}\sum_{n=1}^{N}\bmu(n)f(T^n y)= \frac{1}{N}\sum_{n=1}^{N}\textrm{pr}_1( \Phi \circ \alpha \circ \psi_1^n)(\bmu^2,\omega_1)f(T^n y),
\end{eqnarray}
since $\Phi \circ \alpha \circ \psi_1 = S \circ \Phi \circ \alpha,$ by the definition of $\alpha$ and \eqref{phiformula}. It follows, by \eqref{lespr}, that 

\begin{eqnarray}\label{Jformula3}
\frac{1}{N}\sum_{n=1}^{N}\bmu(n)f(T^n y) & = \frac{1}{N}\sum_{n=1}^{N}(\textrm{Pr}_1 \circ \alpha \circ \psi_1^n)(\bmu^2,\omega_1)f(T^ny),\\
& =\frac{1}{N}\sum_{n=1}^{N}\delta_{\psi_1 \times T((\bmu^2,\omega_1),y)}\big((\textrm{Pr}_1 \circ \alpha) \otimes f )\big)
\end{eqnarray}
Hence, by the compactness of $\mathcal{P}((X_{\A_2} \times \Omega_1) \times Y)$,  we can extract a subsequence $(N_k)$ such that
$$
 \frac{1}{N_k}\sum_{n=1}^{N_k}\delta_{\psi_1 \times T((\bmu^2,\omega_1),y)} \tend{k}{+\infty}{} \xi\,,
 $$
in the weak-star topology.  Denote by $\gamma_1$ and $\gamma_2$ the projections of $(X_{\A_2} \times \Omega_1) \times Y$ on $(X_{\A_2} \times \Omega_1)$ and $Y$, respectively. Then, 
$$
\gamma_1\xi \setdef \lambda \in \mathcal{I}_{\psi_1}(\bmu^2,\omega_1)\,,
$$
and
$$
\gamma_2\xi \setdef \kappa \in \mathcal{I}_{T} (y)\,.
$$
Disintegrating $\xi$ over $\lambda$, we get, for any $A \in \B(X_{\A_2} \times \Omega_1)$ and $B \in \B(Y)$, 
$$
\xi(A \times B)=\int_{A} \kappa_{(x,\omega)}(B) d\lambda(x,\omega)\,,
$$
with   $\kappa_{(x,\omega)} \in \mathcal{P}(Y)$. Moreover, $T\kappa_{(x,\omega)}=\kappa_{{\psi_1}(x,\omega)}$,
since $\psi_1$ is $\lambda$ a.e. invertible. We thus define $\lambda$ a.e. a mapping  
\begin{eqnarray}
& \sigma : & (X_{\A_2}\times \Omega_1)\rightarrow \mathcal{P}(Y)\\
\nonumber & & (x,\omega)\mapsto \sigma(x,\omega)=\kappa_{(x,\omega)}.
\end{eqnarray}
We further have $\sigma \circ \psi_1=T \circ \sigma.$ Put $\Theta=\sigma_*(\lambda)$, that is, $\Theta$ is the pushforward measure of $\lambda$ under $\sigma$. Whence,
$\Theta \in \mathcal{P}(\mathcal{P}(Y))$ is a $T$-invariant probability measure on $\mathcal{P}$. We further have
\begin{claim}The barycenter of $\Theta$ is $\kappa$.
\end{claim}
\begin{proof} We start by writing,
\begin{align}
\int_{\mathcal{P}(Y)} \rho d\Theta(\rho)&=\int_{X_{\A_2} \times \Omega_1} (\Id_{\mathcal{P}(Y)} \circ \sigma)(x,\omega) d\lambda(x,\omega)\\
&=\int_{X_{\A_2} \times \Omega_1} \sigma(x,\omega) d\lambda(x,\omega)\\
&=\int_{X_{\A_2} \times \Omega_1} \kappa_{(x,\omega)} d\lambda(x,\omega)
\end{align}	
Whence, for any $h \in C(Y)$, we have
\begin{align}
\int_{\mathcal{P}(Y)} \int_{Y} h(z)\rho(z) d\Theta(\rho)&=\int_{X_{\A_2} \times \Omega_1} \int_{Y} h(z) d(\kappa_{(x,\omega)})(z) d\lambda(x,\omega)
\end{align}
Put $$ H((x,\omega),z)=h(z), ~~~~~~\forall (x,\omega) \in X_{\A_2} \times \Omega_1, z \in Y.$$
Then 
\begin{align}
\int H((x,\omega),z) d\xi & = \int_{X_{\A_2} \times \Omega_1} \int_{Y} H((x,\omega),z) d\kappa_{(x,\omega)}(z) d\lambda(x,\omega)\\
& = \int_{X_{\A_2} \times \Omega_1} \int_{Y} h(z) d\kappa_{(x,\omega)}(z) d\lambda(x,\omega)\\
& = \int_{Y} h(z) d\kappa(z)
\end{align}
The last equality follows from the fact that $H$ depends only on the $y$ variable. Summarizing, we have proved
$$
\int_{\mathcal{P}(Y)} \int_{Y} h(z)\rho(z) d\Theta(\rho)=\int_{Y} h(z) d\kappa(z)\,,
$$
and the proof of the claim is complete. 
\end{proof}

\smallskip

Therefore $(\mathcal{P}(Y),T,\Theta)$ is a quasi-factor of $(Y,T,\kappa)$. But, since $y$ is
completely deterministic and $\kappa \in I_{T}(y)$, we get the entropy of  the system
$(\mathcal{P}(Y),T,\Theta)$  is zero, by Lemma \ref{GW}. We thus deduce that the
bounded function
$$
F(x,\omega)=\int_{Y} f(z) d\kappa_{(x,\omega)}(z),
$$
is measurable with respect to the Pinsker $\sigma$-algebra
$\Pi_{\lambda}(\psi_1)=\alpha^{-1}\Pi_{\alpha\lambda}(\psi_{0})$.
But, under our assumption, we have
\begin{claim}
$ \ds
\eE(\textrm{Pr}_1 \circ \alpha |_{\Pi_{\lambda}(\psi_1)})=0\,.
$
\end{claim}

\begin{proof} We have
\begin{align}
\eE(\textrm{Pr}_1 \circ \alpha |_{\Pi_{\lambda}(\psi_1)})&=\eE(\textrm{Pr}_1 \circ \alpha |_{\alpha^{-1}\Pi_{\alpha\lambda}(\psi_0)})\\
& = \eE(\textrm{Pr}_1|_{\Pi_{\alpha\lambda}(\psi_0)})\circ \alpha\\
& = \eE(\textrm{pr}_1\circ \Phi|_{\Phi^{-1}\Pi_{\Phi\alpha\lambda}(\psi_0)})\circ \alpha\\
& = \eE(\textrm{pr}_1|_{\Pi_{\Phi\alpha\lambda}(\psi_0)})\circ \Phi \circ \alpha\\
& = 0\,.
\end{align}	
The last equality follows by our assumption. 
\end{proof}

\noindent We thus conclude that 
$$
\int \textrm{Pr}_1 \circ \alpha(x,\omega) F(x,\omega) d\lambda(x,\omega)=\int \eE(\textrm{Pr}_1 \circ
\alpha |_{\Pi_{\lambda}(\psi_1)}) F(x,\omega) d\lambda(x,\omega) =0\,,
$$
that is, zero is the only accumulation point  of the sequence
$$
\frac{1}{N}\sum_{n=1}^{N}\bmu(n)f(T^ny)\,
$$
This completes the proof of the theorem. {\hfill{$\Box$}}

\smallskip

\begin{rem} (1) The famous Chowla conjecture is equivalent to saying that $\bmu$ is
quasi-generic, hence, generic for the Chowla-Sarnak-Veech measure $\eta_M$.
Therefore, if Chowla conjecture holds then our assumption in Theorem \ref{Veech}
is satisfied and hence Chowla conjecture implies Sarnak's M\"{o}bius conjecture.

\noindent (2) We notice that the point $\omega_0$ in the equation \eqref{invmu} is the
Liouville function $\bml$. Furthermore, $\bmu$ is generic for CSV measure is equivalent
to  $(\bmu^2,\bml)$ is generic for $\nu_M \times b(\frac12,\frac12)$. We thus
deduce that Chowla conjecture is equivalent to saying that all the systems
$(X_{\A_3}, \B(\A_3), \eta)$, $\eta \in \mathcal{I}_S(\bmu)$are reduced to just one
dynamical system.

\noindent (3) One might think of formulating a weaker form of Chowla conjecture by
demanding that for $\eta \in \mathcal{I}_S(\bmu)$, all of these dynamical systems
$(X_{\A_3}, \B(\A_3), \eta)$ are measure theoretically isomorphic. But this weak form,
by Sarnak-Veech's theorem (Theorem \ref{SV}), is actually equivalent to Chowla conjecture.
\end{rem}

\smallskip

\noindent The above proof of Veech's theorem actually allows us to view our main theorem
differently. In the following corollary we formulate Theorem \ref{main} in terms of
`spectral isomorphism' and this is where the notation and ideas developed in this section
come into play.

\smallskip

\begin{cor}\label{spect-iso} For each $\eta \in \mathcal{I}_S(\bmu)$, the dynamical system
$(X_{\A_3}, \B(\A_3),\eta ,S)$ is spectrally isomorphic to $(X_{\A_2}\times \Omega_0, \nu_M
\times b(\frac12,\frac12), \psi_0)$.
\end{cor}
\begin{proof} Our proof of Theorem \ref{main} shows that for any $\eta\in \mathcal{I}_S(\bmu)$
the unitary operator on the closed invariant subspace $L^2(X_{\A_3},\Pi (S),\eta)^{\perp}$
has countable Lebesgue spectrum and on $L^2(X_{\A_3},\Pi (S),\eta)$ has discrete spectrum,
independent of $\eta$. Thus for any $\eta\in \mathcal{I}_S(\bmu)$ the system $(X_{\A_3},\B(\A_3),\eta ,S)$
is spectrally isomorphic to $(X_{\A_2}\times \Omega_0, \nu_M \times b(\frac12,\frac12), \psi_0)$. 
\end{proof}

\smallskip

\begin{rem}\
\begin{enumerate}
\item W. Veech's theorem proves that the system $(X_{\A_2}\times \Omega_0, \nu_M \times b(\frac12,\frac12), \psi_0)$
is measure theoretically isomorphic to $(X_{\A_3}, \B(\A_3), \eta_M,S)$. Let us point out also that
the spectral type of  $(X_{\A_2}\times \Omega_0, \nu_M \times b(\frac12,\frac12), \psi_0)$
is given by the sum of discrete measure and a countable Lebesgue component. The discrete
measure is exactly the spectral measure of the rotation on the group $G=\prod_{p \in \pr} \Z/p^2\Z$.

\item  We further notice that if $(X_{\A_3}, \Pi(S), \eta_M,S)$ and  $(X_{\A_3}, \Pi_{\eta}(S), \eta ,S)$
are spectrally isomorphic, then again the hypothesis of Veech conjecture holds and hence by Veech theorem
(Theorem \ref{Veech}), Sarnak conjecture holds. We thus make the following conjecture between Chowla
conjecture and Veech's conjecture.
\end{enumerate}
\end{rem}

\smallskip

\begin{conj} For any $\eta \in \mathcal{I}_S(\bmu)$, dynamical systems $(X_{\A_3}, \Pi(S), \eta,S)$
and  $(X_{\A_3}, \Pi_{\eta}(S), \eta ,S)$ are spectrally isomorphic.
\end{conj}

\smallskip

\noindent{}Considering our result, we also would like to make the following conjecture. For this, we need to recall
the notion of a relative version of $K$-systems.

%\smallskip

%\begin{defn}Let $(X,T)$ be an ergodic dynamical system and let $\mathcal{Q}$ be a generating partition
%of a factor $ \bigvee_{j=-\infty}^{+\infty}~T^{j}(Q)$. We say that this system is a \emph{conditional $K$-system}
%or (relatively $K$-system) with respect to $\bigvee_{j=-\infty}^{+\infty}~T^{j}(Q)$ if, for every partition
%$\mathcal{P}$, the conditions
%\begin{align*}
%(1) &\quad (\bigvee_{j=-\infty}^{+\infty}~T^{j}(P)\supset (\bigvee_{j=-\infty}^{+\infty}~T^{j}(Q)), \\
%(2) &\quad E(\mathcal{P}, T) = E(\mathcal{Q}, T),
%\end{align*}
%imply that
%\[
%(3) \quad \bigvee_{j=-\infty}^{+\infty}~T^{j}(P) = \bigvee_{j=-\infty}^{+\infty}~T^{j}(Q).
%\]
%\end{defn}

\smallskip

%We recall that for any finite partition $R$, $E(T,P)$ is defined by
%$$
%E(T,P)=\lim \frac{1}{N}E(\bigvee_{j=0}^{+k}~T^{-j+1}(P))\,,\quad \text {where}\quad E(R)=-\sum_{A \in R} \mu(A) \log(\mu(A)\,.
%$$
\begin{defn} Let $(X,\B,\mu,T)$ be a dynamical system and $\A$ be an invariant $\sigma$-algebra. $(X,\B,\mu,T)$ is said to be relatively $K$-system with respect to $\A$ if there exist an invariant $\sigma$-algebra $\A_1$ such that 
	$\A$ and $\A_1$ are independent, $\B= \A_1 \vee  \A$, and there exist a partition $\xi$ with $T^n\xi \nearrow \A_1$ and $T^{-n}\xi \searrow \{\emptyset, X\}$.  
\end{defn}
For more details on the relatively $\Kl$-system notion, we refer to \cite{Th1}, \cite{Kamin}.

It is noted in \cite{Th1} that every system is relatively $\Kl$ over its Pinsker factor. 
For a proof of this fact, we refer the reader to \cite{Kamin} (see the proof of (e) on page~22). 
We also warn the reader that the definition of relatively $\Kl$-systems in \cite{Kamin} is more general and precise than the one used in \cite{Th1}.

Finally, we note that the notion of a relative  $\Kl$-system was inspired by the Rokhlin–Sinai machinery (see Theorem 15, Chapter 6 of \cite{P}, and Chapter 13 of \cite{CFS}). In this paper, we extend this framework toward the analysis of the spectral properties of the system arising from number theory.

\smallskip

\noindent \begin{conj}\label{Veech-Sp} For any $\eta \in \mathcal{I}_S(\bmu)$, the dynamical system
$(X_{\A_3}, \B(\A_3) ,\eta,S)$ is relatively $K$-system with respect to the Pinsker factor of the
CSV measure. That is, it is relatively $K$-system with respect to the sigma algebra $\Pi(S)$.
\end{conj}
%\end{enumerate}

\smallskip
\begin{rem}
Notice that if the conjecture \ref{Veech-Sp} holds, then our spectral decomposition reduce to 
$$
L^2(X_{\A_3}, \B(\A_3),\eta)=L^2(X_{\A_3},\Pi(S),\eta)\oplus L^2(X_{\A_3},\Pi (S),\eta)^{\perp}.
$$
This simplifies the proof, since we need to apply Rokhlin's theorem only once.
 
One may ask if this latter conjecture is implied by Veech's conjecture. The answer is yes, since Chowla
and Sarnak M\"{o}bius conjectures are equivalent by the main result in \cite{AV}.
\end{rem}

\medskip

%\appendix

%\noindent {\bf Appendix By Michael Lin}

%\medskip

\textbf{Acknowledgment.}
The first author would like to thank Michael Lin for his invitation and for a discussion on the subject.
He gratefully thanks Ben Gurion University of the Negev, Center of advanced studies in Mathematics
of Ben-Gurion and the Institute of Mathematical Sciences at Chennai for their hospitality, where
part of this work was done. The present work was revisited during a most enjoyable visit to Rutgers University in 2024, and the first author would like to express his gratitude for the generous support and warm hospitality received. 
%Finally, the authors wish to thank Professor Michael Lin for permitting the inclusion of Subsection \ref{Lin}.
Finally, the authors wish to thank Prof. Michael Lin
for graciously permitting us to include a part of his work in subsection \ref{Lin}, which relies on our main theorem.


\begin{thebibliography}{99}
	%{\small
\bibitem{A}
e. H. el Abdalaoui, On the spectrum of $\alpha$-rigid maps,
J. Dyn. Control Syst. 15 (2009),  453-470.


\bibitem{AD}
e. H. el Abdalaoui \& M. Disertori,  Spectral properties of the M\"{o}bius function and a random M\"{o}bius model. Stoch. Dyn. 16 (2016), no. 1, 1650005, 25 pp.
	
\bibitem{ALR}
e. H. el Abdalaoui , M. Lemanczyk \& T. de la Rue, Automorphisms with quasi-discrete spectrum, multiplicative functions and average orthogonality along short intervals, IRMN, 14 (2017), 4350–4368.

\bibitem{AV}
e. H. el Abdalaoui, \emph{On Veech's proof of Sarnak's theorem on the M\"{o}bius flow},
{\it Preprint}, 2017, arXiv:1711.06326 [math.DS].
	
\bibitem{AM}
e. H. el Abdalaoui and M. Nerurkar, \emph{ Weakly tame systems, their characterizations and applications}, Monat. Math., 201(3) (2023), 1-45.

\bibitem{AL}
E.H. el Abdalaoui and M. Lin, Operator ergodic theorems with M\"obius "weights", preprint, 2025.

%\bibitem{Al}
%V.M. Alexeyev,  Existence of a bounded function of the maximal spectral type, Translated from
% the 1958 Russian original by A. Katok, Ergodic Theory Dynam. Systems 2 (1982),  259-261 (1983)
 
\bibitem{B} R.J. Baxter, {\it A class of ergodic automorphisms}, Ph. D. Thesis, 
University of Toronto, 1969. https://www.proquest.com.

\bibitem{BeL}
A. Bellow and V. Losert, The weighted pointwise ergodic theorem and the individual ergodic theorem along subsequences, TAMS v. 288, no. 1 (1985), 307-345.
	
\bibitem{Bert}
J. P. Bertrandias, Espaces des fonctions born\'ees et continues en moyenne asymptotique d'ordre p, Bull. Soc. Math. France 5 (1966), 1-106.
	
\bibitem{Bhatta}
A. Bhattacharyya, \emph{On a measure of divergence between two statistical populations defined by their probability distributions}, Bulletin of the Calcutta Mathematical Society 35 (1943), 99--109.
	
\bibitem{BSZ}
J. Bourgain, P. Sarnak, and T. Ziegler. Disjointness of M\"{o}bius from horocycle flows. In From
Fourier analysis and number theory to radon transforms and geometry, volume 28 of Dev.  Math., p. 67–83. Springer, New York, 2013.
	
\bibitem{CS}
 F. Cellarosi \& G. Ya. Sinai,  Ergodic properties of square-free numbers. J. Eur. Math. Soc. (JEMS) 15 (2013), no. 4, 1343-1374.

 \bibitem{CFS}
I. P. Cornfeld; S. V. Fomin;Ya. G. Sinaĭ,
Ergodic theory.
Translated from the Russian by A. B. Sosinskiĭ
Grundlehren Math. Wiss., 245[Fundamental Principles of Mathematical Sciences]
Springer-Verlag, New York, 1982. x+486 pp.

\bibitem{CMFK}
J. Coquet, T. Kamae, M. Mend\`es-France, \emph{Sur la mesure spectrale de certaines suites arithm\'etiques},
Bull. Soc. Math. France, 105 (1977), 369--384.
	
\bibitem{chowla}
S.~Chowla, \emph{The Riemann hypothesis and Hilbert's tenth problem}, Mathematics and Its Applications, Vol 4, Gordon and Breach Science Publishers, New York, 1965. 
	
\bibitem{Dabo}
H. Daboussi, Fonctions multiplicatives presque p\'eriodiques B. (French) D'apr\`es un travail commun avec Hubert Delange. Journ\'ees Arithm\'etiques de Bordeaux (Conf., Univ. Bordeaux, Bordeaux, 1974), pp. 321–324. Ast\'erisque, No. 24-25, Soc. Math. France, Paris, 1975.
	
\bibitem{Da}
H. Davenport, \emph{On some infinite series involving arithmetical functions. II}, Quart. J. Math. Oxf. \textbf{8} (1937),  313--320.

\bibitem{DGS}
 M. Denker, C. Grillenberger, K. Sigmund, \emph{Ergodic Theory on compact spaces}, Lecture Notes in Mathematics 527, Springer Verlag Pbl. 
 	
\bibitem{G}
E. Glasner, Ergodic theory via joinings. Mathematical Surveys and Monographs, 101. American Mathematical Society,Providence, RI, 2003. xii+384 pp.
	
\bibitem{HPM}
B. Host, J.-F. M\'ela and F. Parreau, Analyse harmonique des mesures [Harmonic analysis of measures],
Ast\'erisque, No. 135-136 (1986), 261 pp. 
	
\bibitem{HWZ}
 W. Huang, Z. Wang; G. Zhang,  M\"{o}bius disjointness for topological models of ergodic systems with discrete spectrum.
J. Mod. Dyn. 14 (2019), 277-290. 
 
 \bibitem{HZY}
 W. Huang, Z. Wang and X. Ye, Measure complexity and M\"{o}bius disjointness, Adv. Math. 347 (2019), 827–858. 

%\bibitem{K}
%K. Peterson and J-P. Thouvenot, Tall fields generated by symbol counts in measure-preserving systems, Colloquim Math., VOL. 101, 2004, 9-22.

 \bibitem{Kamin}
 B. Kami\'{n}ski, On regular generators of $\Z^2$-actions in exhaustive partitions. Studia Math. 85 (1986), no. 1, 17–26 (1987).

\bibitem{KLR}
A. Kanigowski, M. Lema\'{n}czyk and M. Radziwi{\l}{\l}, Rigidity in dynamics and M\"{o}bius disjointness,
Preprint (2019), arXiv:1905.13256.

\bibitem{KN}
L. Kuipers and H. Niederreiter, Uniform distribution of sequences, Wiley-Interscience, New York,
1974.

\bibitem{LGE}
L. Ge, Topology of natural numbers and entropy of arithmetic functions, in Operator Algebras and
Their Applications: A Tribute to Richard V. Kadison, Contemporary Mathematics, vol. 671 (Amer-
ican Mathematical Society, Providence, RI, 2016), 127-144.

\bibitem{MRT}
K. Matom{\"a}ki, M. Radziwi{\l}{\l} and T. Tao, An averaged form of {C}howla's conjecture, Algebra Number Theory,
Vol (9), 2015, 2167-2196.

\bibitem{MR}
K. Matom{\"a}ki, M. Radziwi{\l}{\l}, Multiplicative functions in short intervals. Annals
of Math. 183 (2016) 1015-1056.

\bibitem{Matusita1}
K. Matusita, \emph{Decision rules, based on the distance for problems of fit, two samples, and estimation},
Ann. Math. Statist. 26 (1955), 631--640.

%
\bibitem{Matusita2}
K. Matusita, \emph{A distance and related statistics in multivariate analysis,} Multivariate Analysis
(Proc. Internat. Sympos., Dayton, Ohio, 1965), 1966, pp. 187--200.
%

\bibitem{Matusita3}
K. Matusita, \emph{ On the notion of affinity of several distributions and some of its applications,}
Ann. Inst. Statist. Math. 19 1967 181-192.

\bibitem{Ms}
L. Mirsky, Arithmetical pattern problems relating to divisibility by $r$-th powers, Proc. London Math. Soc. (2) 50, (1949). 497-508.

\bibitem{MV}
 R. Murty and A. Vatwani, A remark on a conjecture of Chowla. J. Ramanujan Math. Soc. 33 (2018), no. 2, 111-123. 
 
 \bibitem{N}
M. G. Nadkarni, {\it Spectral theory of dynamical systems}, Springer, Singapore; Hindustan Book Agency, New Delhi, [2020]. xiii+223 pp.


\bibitem{P}
W. Parry, \emph{Topics in ergodic theory,} Reprint of the 1981 original. Cambridge Tracts in Mathematics,
75. Cambridge University Press, Cambridge, 2004.

\bibitem{R}
G. Rauzy, \emph{Propri\'et\'es statistiques de suites arithm\'etiques}, PUF, 1976.

\bibitem{Rh}
V. A. Rokhlin, Lectures on the entropy theory of transformations
with invariant measure. Usp. Mat. Nauk. 22 (1967), 3-56 (Russian)-Russian Math. Surveys 22 (1967),1-52.

\bibitem{Sarnak}
P. Sarnak, M\"{o}bius randomness and dynamics. Not. S. Afr. Math. Soc. 43 (2012), no. 2, 89-97. 

%\bibitem{Th}
%J-P., Thouvenot, Une classe de syst\`emes pour lesquels la conjecture de Pinsker est vraie. (French) Israel J. Math. 21 (1975), no. 2-3, 208–214. 

\bibitem{Th1}
J.-P. Thouvenot, Une classe de syst\`emes pour lesquels la conjecture de Pinsker est
vraie, Israel J. Math. 21 (1975), 208–214.

\bibitem{VeechNotes2}
W. Veech, M\"{o}bius dynamics, Lecture Notes, Spring Semester 2016, +164 pp, private.

\bibitem{Veech-preprint} W. Veech, A Conjecture Between the Chowla and
Sarnak Conjectures, preprint, May 6, 2016, private communication.  

\bibitem{W}
P. Walters, An introduction to ergodic theory. Graduate Texts in Mathematics, 79. Springer-Verlag, New York-Berlin, 1982. 

\bibitem{Wiener}
N. Wiener, \emph{The Fourier integral and certain of its applications,} Dover publications, (1958).

\bibitem{FW}
F. Wei, Disjointness of M\"{o}bius from asymptotically periodic functions, Pure Appl. Math. Q. 18, No. 3, 863-922 (2022). arXiv:1810.07360v5 [math.NT].	
\end{thebibliography}
\end{document}